\documentclass{amsart}
\usepackage{amssymb}   
\usepackage{amsmath}
\usepackage{mathtools}
\usepackage{amsthm}
\usepackage[all,dvips]{xy}
\usepackage[]{pstricks,pst-node,pst-plot}
\usepackage{graphicx}
\usepackage{pstcol,pst-text}
\def\a={{\buildrel a \over =}}
\def\na={{\buildrel a \over \neq}}
\def\CC{{\mathbb C}}
\def\EE{{\mathbb E}}
\def\PP{{\mathbb P}}

\def\QQ{{\mathbb Q}}

\def\ZZ{{\mathbb Z}}

\def\H{{\mathcal H}}
\def\M{{\mathcal M}}

\def\tH{\widetilde{\mathcal H}}
\def\tC{\widetilde{\mathcal C}}
\def\tP{\widetilde{\PP}}

\def\bM{\overline{\mathcal M}}
\def\bM{\overline{\M}}

\newtheorem{theorem}{Theorem}[section]
\newtheorem{lemma}[theorem]{Lemma}
\newtheorem{proposition}[theorem]{Proposition}
\newtheorem{corollary}[theorem]{Corollary}

\newtheorem{definition-lemma}[theorem]{Definition-Lemma}

\newtheorem{conclusion}[theorem]{Conclusion}
\theoremstyle{definition}
\newtheorem{definition}[theorem]{\bf Definition}
\newtheorem{example}[theorem]{\bf Example}

\theoremstyle{remark}
\newtheorem{remark}[theorem]{\bf Remark}
\newtheorem{assumption}[theorem]{\bf Assumption}
\newtheorem{diagram}[theorem]{\bf Diagram}

\newtheorem{question}[theorem]{\bf Question}

\def\vandaag{\number\day\space\ifcase\month\or
 januari\or februari\or  maart\or  april\or mei\or juni\or  juli\or
 augustus\or  september\or  oktober\or november\or  december\or\fi,
\number\year}
\def\today{\ifcase\month\or
 Jan\or Febr\or  Mar\or  Apr\or May\or Jun\or  Jul\or
 Aug\or  Sep\or  Oct\or Nov\or  Dec\or\fi
 \space\number\day, \number\year}

\def\llongrightarrow{\relbar\joinrel\relbar\joinrel\rightarrow}

\begin{document}

\title[Divisorial Maroni Loci]
{The Cycle Classes of Divisorial Maroni Loci}
\author{Gerard van der Geer $\,$}
\address{Korteweg-de Vries Instituut, Universiteit van
Amsterdam, Postbus 94248, 1090 GE Amsterdam,  The Netherlands}
\email{G.B.M.vanderGeer@uva.nl}
\author{$\,$ Alexis Kouvidakis}
\address{Department of Mathematics and Applied Mathematics, University of Crete,
GR-70013 Heraklion, Greece}
\email{kouvid@uoc.gr}
\begin{abstract}
We determine the cycle classes of effective divisors in the compactified 
Hurwitz spaces $\overline{\mathcal H}_{d, g}$ of curves of genus 
$g$ with a linear system of degree~$d$, that extend the Maroni
divisors on ${\mathcal H}_{d,g}$.
Our approach uses Chern classes associated to a global-to-local
evaluation map of a vector bundle over a generic ${\PP}^1$-bundle
over the Hurwitz space.
\end{abstract}
\maketitle
\begin{section}{Introduction}\label{sec:intro}
Let $C$ be a smooth projective curve of genus $g$ and let $\gamma: C \to {\PP}^1$
exhibit $C$ as a simply-branched cover of degree $d$ of the projective line.
Then we can associate to~$\gamma$ a $(d-1)$-tuple $\alpha=(a_1,\ldots,a_{d-1})$ of integers
with $\sum_{i} a_i= g+d-1$  by considering the kernel of the trace map
$\gamma_* {\mathcal O}_C \to {\mathcal O}_{{\PP}^1}$; this is a locally free sheaf
of rank $d-1$
on ${\PP}^1$; hence by Birkhoff-Grothendieck 
its dual can be written as $\oplus_{i} {\mathcal O}(a_i)$,
where we may assume that the $a_i$ are listed in non-decreasing order.
For sufficiently general $C$ and when $d<g$ this $(d-1)$-tuple $\alpha=(a_1,\ldots,a_{d-1})$
has a geometric interpretation 
as the type of the scroll in the canonical space ${\PP}^{g-1}$ of $C$
with fibres the ${\PP}^{d-2}$ spanned by the fibres 
of $\gamma$. The invariant $\alpha$ allows one
to define a stratification on the Hurwitz space ${\mathcal H}_{d,g}$ 
of genus $g$ degree $d$ simply-branched covers of ${\PP}^1$, with the strata
corresponding to these $(d-1)$-tuples $\alpha$. 

In the case of trigonal
covers this stratification was studied by Maroni (see \cite{Maroni})
and the stratification
is named after him. 
The numbers $a_i$  are sometimes called the {\sl scrollar
invariants} of the cover $\gamma$ and seem to appear first in a paper
by Christoffel dating from 1878 (see \cite[p.\ 242--243]{Christoffel}) and 
alternatively can be characterized as the
numbers $n$ where the function $h^1((n-1)D)-h^1(n\, D)$ jumps with $D$ a
divisor of the linear system given by $\gamma$. These invariants attracted 
attention by the work of Schreyer \cite{Schreyer}. 
A result of Ballico~\cite{Ballico}
implies that for a general cover $\gamma$  of genus 
$$
g=(d-1)k+s \qquad \text{with} \quad 0\leq s < d-1
$$
the $(d-1)$-tuple $\alpha$ is of the form $(k+1,k+1,\ldots,k+1,k+2,\ldots,k+2)$
with $d-1-s$ terms equal to $k+1$. If $s=0$ one expects that there
is a divisor where the type $\alpha$ jumps and this holds indeed 
as was proved by
Coppens and Martens in \cite{C-M}; see also \cite{Patel} and \cite[p.\ 13/14]{F-L}.
For an overview of the theory of these Maroni loci we refer to the
thesis of Patel \cite{Patel,Patel-ArXiv}.

These strata $M_{\alpha}$ are divisors only if $g$ is a multiple of $d-1$,
say $g=(d-1)k$ and then $\alpha=(k,k+1,\ldots,k+1,k+2)$. 
It is natural to ask for the
divisor classes of the closures of these Maroni loci in the admissible cover
compactification $\overline{\mathcal H}_{d,g}$ introduced by Harris and Mumford
in~\cite{HM}.
This question was studied by Stankova-Frenkel in \cite{St-F} for the trigonal
case and for degree $4$ and $5$ by Patel in his Harvard thesis \cite{Patel}.
Patel determined classes in a partial compactification of the Hurwitz space. 
Deopurkar and Patel calculated the class of the closure of the Maroni divisor
in the trigonal case in \cite{D-P}. 

In this paper we consider the general case and we construct
certain effective divisors containing the closure of the Maroni divisor
and calculate their classes.
We do this in a series of steps. 
We start by constructing a good model for the universal $d$-gonal map
$\gamma: {\mathcal C} \to {\PP}^1_{{\mathcal H}_{d,g}}$, where a good model
stands for a proper flat and finite 
map $\tilde{\gamma}: Y \to \tilde{\PP}$ that extends
$\gamma$ over a normalization $\widetilde{\mathcal H}_{d,g}$ of 
$\overline{\mathcal H}_{d,g}$. This allows us to extend the vector bundle 
$V$ over  ${\PP}^1_{{\mathcal H}_{d,g}}$ 
used to define the Maroni invariant $\alpha$ to a reflexive sheaf 
over the compactification $\widetilde{\PP}$ and ignoring codimension $\geq 3$ loci
we may consider it as a vector bundle.
However, since the fibres of the projection  
$p:\widetilde{\PP} \to \overline{\mathcal{H}}_{d,g}$ 
of the compactified space $\widetilde{\PP}$
are no longer always isomorphic to ${\PP}^1$, but can be 
chains of ${\PP}^1$, we cannot directly extend the definition
of $\alpha$. Instead we extend the Maroni divisor as follows.

First we twist the extended vector bundle $V$ by an explicit 
line bundle so that we get a vector bundle $V'$ on $\widetilde{\PP}$
which  restricted to a general fibre of $p$ is trivial. 
We associate a divisor class on $\widetilde{\PP}$ 
to the global-to-local evaluation map 
$$
{\rm ev}: p^*p_*V' \longrightarrow V'
$$ 
and show that this divisor
class is represented by an effective divisor  
whose pullback under a section
of $\widetilde{\PP}$ gives an effective divisor 
$\mathfrak{m}_{\rm st}$ on $\widetilde{\mathcal H}_{d,g}$ 
that extends the Maroni divisor on ${\mathcal H}_{d,g}$.
Our idea of using the morphism $p^*p_*V' \to V'$ was inspired 
by a paper of Brosius \cite{Brosius}.

The vector bundle $V$ (or rather reflexive sheaf) 
that we use is an extension of the dual of the cokernel 
of the natural map ${\mathcal O}_{{\PP}^{1}_{{\mathcal H}_{d,g}}} \to \gamma_*
{\mathcal O}_{\mathcal C}$. It is obtained by taking the structure sheaf
on $Y$ and taking the dual $V$ of the cokernel of the natural map 
$\iota: {\mathcal O}_{\widetilde{\PP}} \to \tilde{\gamma}_* {\mathcal O}_{Y}$
and by twisting in a standard way 
$V'=V \otimes M$, where $M$ is an explicitly defined line bundle on $\widetilde{\PP}$.
With these choices our construction 
yields an explicit effective divisor class $\mathfrak{m}_{\rm st}$
that extends the Maroni divisor. We express the class $\mathfrak{m}_{\rm st}$
explicitly in terms of the classes of boundary divisors of
the compactified  Hurwitz space. 

But we can replace 
${\mathcal O}_{Y}$ by any effective line bundle
${\mathcal L}$ on $Y$ 
that extends the trivial line bundle on ${\mathcal C}$ and instead take
the dual $V_{\mathcal{L}}$ of the cokernel of the natural map 
$\iota: {\mathcal O}_{\widetilde{\PP}} \to \tilde{\gamma}_* {\mathcal L}$. 
And then after twisting by $M$ we can twist by an arbitrary line bundle 
$N$ corresponding to a divisor with support on the
boundary of $\widetilde{\PP}$, 
that is, the singular fibres of the generic ${\PP}^1$-fibration,
as the result remains trivial on the generic fibre of $p$.
Doing this produces variations $\mathfrak{m}_{\mathcal{L},N}$ of
$\mathfrak{m}_{\rm st}$ each of which is a class of an effective divisor that
extends the Maroni locus. In particular, with $\mathcal{L}$ trivial varying $N$
allows us to get rid of the specific choice of $M$ to make $V$
trivial on the generic fibre.

We study the effect of changing ${\mathcal L}$ and $N$ on the resulting
effective divisors $\mathfrak{m}_{{\mathcal L},N}$. By appropriately
choosing ${\mathcal L}$ and $N$ we often 
can reduce the divisor $\mathfrak{m}_{\rm st}$ 
by removing superfluous divisors supported on the boundary. 
In fact, the effect of twisting is quadratic and we can find the critical
points of this quadratic function. But these critical 
points correspond to elements
in the rational Picard group and we thus have to approximate by integral
points.
If we take $\mathcal{L}$ trivial and vary $N$ the result is a completely 
explicit combinatorial formula 
(Thm.\ \ref{correction1})
for this effective class in terms of the boundary classes of the compactified
Hurwitz space and it involves only local contributions. 
If we vary also $\mathcal{L}$ it involves global classes (Thm.\ 
\ref{differenceLN}), but under
certain assumptions the contributions still can be given explicitly
(Thm.\ \ref{correction2}) and then finding the optimal choices
of $(\mathcal{L},N)$ reduces to a combinatorial problem.

The rational Picard group of the compactified Hurwitz space is conjectured to be
generated by the classes of the irreducible boundary components.
In any case, we work inside the subgroup generated by these boundary
components.
We can define an effective divisor containing the Maroni class
by
$$
\mathfrak{m}_{\min}= \cap_{\mathcal{L},N} \mathfrak{m}_{\mathcal{L},N}\, ,
$$
where $\mathcal{L}$ and $N$ vary.
We do not know whether $\mathfrak{m}_{\min}$ coincides with the closure of the
Maroni locus, but in the trigonal case ($d=3$ and $g$ even) it does:
the class of the closure of the Maroni 
locus coincides with $\mathfrak{m}_{\mathcal{L},N}$ for suitable 
$(\mathcal{L},N)$. In fact,
in \cite{D-P} Deopurkar and Patel determined the class of the 
Zariski closure of 
the Maroni divisor in the trigonal case and we show that an appropriate choice
of ${\mathcal L}$ and $N$ reproduces the class found by Deopurkar and Patel. 
We also show that the effective class found by Patel in \cite{Patel} 
containing the Maroni locus 
on a partial compactification is in general larger than the class of the Zariski 
closure of the Maroni locus.
We point out that we do not use the standard tool of 
test curves that seems inadequate here. 

In a subsequent paper (\cite{GK_A}) 
we construct analogues of the Maroni stratifications
for the case where the genus $g$ is not divisible by $d-1$ 
that possess strata that are divisors. 
The methods in this paper extend to that case.

We hope that this method to determine the class of an effective divisor
extending the Maroni divisor is of independent interest for studying cycle
classes on moduli spaces. 

We finish with a remark on the terminology. We will encounter reflexive sheaves,
but as these are locally free outside a locus of codimension $\geq 3$ on smooth 
spaces, the fact that these are not necessarily locally free  will
not affect our divisor class calculations and frequently we will treat these
reflexive sheaves as if they were vector bundles. Finally, we will often use
the same symbol for a divisor and its divisor class or cohomology class.

\noindent{\sl Acknowledgements.} The authors thank G.\ Martens for the reference to
Christoffel and they thank the referees for their careful reading and their 
remarks. 
The first author thanks YMSC of Tsinghua University for the hospitality enjoyed there.
\end{section}
\begin{section}{Maroni Loci}
Let ${\mathcal H}_{d,g}$ be the Hurwitz space of genus $g$ degree $d$ 
simply-branched covers of ${\PP}^1$ over the field of complex numbers
with ordered branch points on ${\PP}^1$. 
Here simply-branched means that every fibre contains at least $d-1$ points.
It is a Deligne-Mumford stack of 
dimension $2g+2d-5$. 
To a simply-branched cover $\gamma: C \to {\PP}^1$ we can associate
an exact sequence of sheaves 
$$
0 \to {\mathcal E}_1 \to \gamma_* {\mathcal O}_{C} 
{\buildrel {\rm tr} \over \llongrightarrow }  {\mathcal O}_{{\PP}^1}\to 0\, ,
$$
where ${\rm tr}$ stands for the trace map and 
${\mathcal E}_1=\ker({\rm tr})$. Since ${\rm tr}$ 
is multiplication by the degree $d$ on the pull back of 
${\mathcal O}_{{\PP}^1}$ under $\gamma$, the sequence splits
and ${\mathcal E}_1$ is locally free as a direct summand 
of a locally free sheaf.
Let ${\mathcal E}={\mathcal E}_1^{\vee}$ be the 
dual locally free sheaf, which we
view as a vector bundle $V$ of rank $d-1$.  
As a bundle on ${\PP}^1$ it can be written as
a direct sum $\oplus_{i=1}^{d-1} {\mathcal O}(a_i)$ 
with $\sum_{i=1}^{d-1} a_i=
g+d-1$. For a given cover~$\gamma$ we can look at the type, 
that is, the $(d-1)$tuple
$\alpha=(a_1,\ldots, a_{d-1})$, where we may and will assume that 
$a_1\leq a_2 \leq \cdots \leq a_{d-1}$. 
Such a type defines a locus ${\mathcal M}_{\alpha}$ in ${\mathcal H}_{d,g}$, 
namely the Zariski closure of the set of covers $\gamma$ 
whose associated $V$ has type~$\alpha$.
We are interested in the case where such a Maroni locus is a divisor.
According to a result by Coppens and Martens (see \cite{C-M}) and Patel \cite[Thm 1.13 and Thm 1.15]{Patel} 
we find that this happens 
if and only if the genus is a multiple of $d-1$, say
$$
g=(d-1)\, k ,
$$
and that then generically we have $\alpha=(k+1,\ldots,k+1)$, 
that is, $V$ is {\sl balanced}, while 
for a point on this divisor generically we have 
$\alpha=(k,k+1,\ldots,k+1,k+2)$. Moreover, the divisor is
irreducible.
Thus for covers $\gamma: C \to {\PP}^1$ with $g=(d-1)k$ not contained 
in the Maroni locus ${\mathcal M}_{\alpha}$ with
$\alpha=(k,k+1,\ldots,k+1,k+2)$ the 
vector bundle $V$ is balanced, while generically on the Maroni locus
we have a minimal deviation from this behavior, 
see \cite[Thm.\ 1.15]{Patel}.  

We are interested in the cycle classes of this Maroni divisor. 
By a ${\PP}^1$-fibration we mean a morphism $p: {\PP}\to X$,
where ${\PP}$ is the projectivization of a rank $2$ vector bundle
on $X$. In our case we have a ${\PP}^1$-fibration ${\PP}$
over our base together with a vector bundle $V$ on ${\PP}$.
We may twist the vector bundle $V$ by tensoring by 
${\mathcal O}(-k-1)$, that is, by ${\mathcal O}((-k-1)S)$, with
$S$ the image of a section of the ${\PP}^1$-fibration. 
Carrying out the above construction in families we arrive 
at the situation, where we have a vector bundle $V$ of
rank $r$ on a ${\PP}^1$-fibration ${\PP}$ over ${\mathcal H}_{d,g}$ that
is trivial on the generic fibre ${\PP}^1$ and has minimal deviation
${\mathcal O}(-1)\oplus {\mathcal O} \oplus \cdots 
{\mathcal O}\oplus {\mathcal O}(1)$
on a divisor in ${\mathcal H}_{d,g}$ 
(and larger deviations in higher codimension only).
This will imply that the support of $R^1p_*V$ has codimension $>1$.

So suppose that we have a ${\PP}^1$-fibration $p:{\PP} \to X$, with
$X$ a smooth base, 
and a vector bundle $V$ of rank $r$ (or more generally a reflexive sheaf)
on ${\PP}$ which is trivial on the generic fibre.
We consider the canonical global-to-local map of sheaves on ${\PP}$
$$
{\rm ev} : p^*p_*V \longrightarrow V
$$
given by evaluation. We note that $p_*V$ is a reflexive sheaf (isomorphic to
its double dual) by \cite[Corollary 1.7]{Hartshorne}.
Therefore this sheaf is locally free outside a locus of codimension 
$\geq 3$. We will work outside this locus and then may assume that
$p_*V$ is locally free. If $U$ is an open subset of $X$ over which 
$p_*V$ is free of rank $r$ we can choose $r$ generating sections 
$s_1,\ldots,s_r$ of $p_*V$ on $U$ and consider their pullbacks $p^*s_i$ to~${\PP}$. 
The images of these $s_i$ under the restriction map
$$
p^*p_* V \to H^0(p^{-1}(x), V|p^{-1}(x))
$$
to the fibre $p^{-1}(x)={\PP}^1$ over $x$ 
for $x \in U$ generate  a subspace of $H^0({\PP}^1, V|p^{-1}(x))$, 
but if $V|p^{-1}(x)$
has type $\oplus_{i=1}^r {\mathcal O}(a_i)$ with at least one 
$a_i<0$ these sections cannot
generate the stalk $\oplus_{i=1}^r {\mathcal O}(a_i)$, since these sections
are located in the `non-negative part'
$$
\oplus_{a_i \geq 0} {\mathcal O}(a_i)\, .
$$
Therefore the global-to-local map ${\rm ev}: p^*p_*V \to V$, 
a map of vector bundles of the same rank, must have vanishing 
determinant. On the other hand if $V|p^{-1}(x)$ is trivial of 
rank $r$, then by Grauert's theorem the $r$ sections $p^*s_i$, which are 
linearly independent in 
$H^0(p^{-1}(x), V|p^{-1}(x))= H^0({\PP}^1, {\mathcal O}^r)$,
necessarily generate $V|p^{-1}(x)$.
We summarize as follows.

\begin{proposition}\label{detvsMaroni}
Outside a locus of codimension $\geq 3$ the support of the vanishing locus
of the determinant of the homomorphism 
${\rm ev}: p^*p_*V \to V$ coincides with the inverse
image under $p$ of the Maroni locus.
\end{proposition}

\bigskip
We are interested in the first Chern class $c_1(Q)=c_1(V)-c_1(p^*p_*V)$
of the degeneracy locus $Q$ of the evaluation map ${\rm ev}: p^*p_*V \to V$.
Since the map ${\rm ev}$ is injective in codimension $1$ this class coincides
with the first Chern class of the cokernel of the map ${\rm ev}$. 
This is a class on ${\PP}$, but in accordance with Proposition 
\ref{detvsMaroni}  it is in fact a pull back class from $X$ as we shall
see now.

\begin{lemma}
If $V$ is a vector bundle on a ${\PP^1}$-fibration $p:{\PP} \to X$ which is
trivial on the generic fibre and has minimal deviation 
$(-1,0,\ldots,0,1)$ over a divisor in $X$, then the divisor class of the determinant 
of ${\rm ev} : p^*p_*V \to V$ is given by 
$$
c_1(Q)=c_1(V)-c_1(p^*p_*V)=p^*p_* c_2(V)\, .
$$
\end{lemma}
\begin{proof}
The sheaf $R^1p_*V$ has support in codimension $\geq 2$, therefore we can
neglect it for the calculation of $c_1(p_*V)$. Let $\theta_p$ be the first
Chern class of the relative dualizing sheaf $\omega_p$ of $p$. 
Then the Grothendieck-Riemann-Roch Theorem applied to $p$ and $V$ gives us
$$
\begin{aligned}
c_1(p_*V)&=p_*[{\rm ch}(V) \, {\rm td}(T_{p})]_{(2)} \\
&=p_*[(r+c_1(V)+c_1^2(V)/2-c_2(V))(1-\theta_p/2+\theta_p^2/12)]_{(2)} \\
&=p_*[r\, \theta_p^2/12-c_1(V)\theta_p/2+ c_1^2(V)/2-c_2(V)] \, . \\
\end{aligned}
$$
Now the class $c_1(V)$ restricted to the fibres is trivial 
(because of degree $0$ on a ${\PP}^1$), 
hence $c_1(V)$ is a pullback $c_1(p^*L)$ for some line bundle $L$ 
on $X$. If $\sigma: X \to {\PP}$ is a section with image $S$,
then $\omega_p\cong {\mathcal O}(-2S) \otimes p^*\sigma^*{\mathcal O}(S)$, 
and therefore
$$
\theta_p^2= 4 \, S^2 -4 \, S \cdot p^*\sigma^* S+ p^*\sigma^*S^2\, .
$$
We observe 
$p_*p^*\sigma^*S^2=0$, $S^2=\sigma_*\sigma^*S$, $S\cdot p^*\sigma^*S=
\sigma_*\sigma^*p^*\sigma^*S$ by the projection formula, and
so $\sigma_*\sigma^*p^*\sigma^*S=\sigma_*\sigma^*S$ because $p\, \sigma=1$.
So $p_{*} \theta_p^2= 4\, \sigma_*\sigma^* S-4\,  \sigma_*\sigma^*S=0$.
Furthermore, we have
$$
p_*(c_1(V)\cdot \theta_p)=p_*(p^*c_1(L)\cdot (-2\, S+p^*\sigma^*S))
=-2 \, c_1(L) \cdot p_*S=-2\, c_1(L)\, .
$$
We conclude that $c_1(p_*V)=c_1(L)-p_*c_2(V)$, hence 
$[\det({\rm ev})]=c_1(V)-p^*c_1(p_*V)=p^*p_{*}c_2(V)$.
\end{proof}

\begin{corollary}
The pullback of $c_1(Q)$ 
under a section $\sigma: X \to {\PP}$ of $p$ gives an 
effective
divisor with support on the locus over which $V$ is not trivial.
Its class is equal to $p_*c_2(V)$.
\end{corollary}
\begin{proof}
The pullback is $\sigma^*p^*p_*c_2(V)=p_*c_2(V)$.
\end{proof}

If we do not assume that $R^1p_{*}V$ is trivial in codimension $1$ then 
we can reformulate the result as follows.

\begin{lemma}
If $V$ is a vector bundle on a ${\PP}^1$-fibration $p:{\PP} \to X$ which is trivial
on the generic fibre, 
then we have the identity
$$
c_1(Q)+p^* c_1(R^1p_*V)=p^{*}p_{*}c_2(V)\, .
$$
\end{lemma}

Since in general we have a bundle $V$ on a ${\PP}^1$-fibration 
which is balanced of type $(k,\ldots,k)$ on the generic fibre, 
but not trivial on the generic fibre, we check what happens
if we twist by a line bundle that on the generic fibre is a power of 
${\mathcal O}_{{\PP}^1}(1)$.
The canonical global-to-local map $p^*p_*(V(-k))\to V(-k)$ 
induces now a canonical map
$(p^*p_*V(-k))(k) \to V$ with degeneracy locus $Q$.

\begin{lemma}\label{firstresult}
If $V$ is a vector bundle of rank $r$ on a ${\PP^1}$-fibration $p:{\PP} \to X$ which is
balanced on the generic fibre and has minimal 
deviation $(k-1,k,\ldots,k,k+1)$ over a divisor in $X$, 
then the pull back of the first Chern class of the degeneracy locus $Q$ of
${\rm ev} : (p^*p_*V(-k))(k) \to V$ under $\sigma^*$ is given by the
effective class 
$$
c_1(Q)= p_*\left( c_2(V)-\frac{r-1}{2\, r} c_1^2(V)\right) .  \eqno(1)
$$
\end{lemma}
\begin{proof}
Let $c_1({\mathcal O}(-k))=-kS$ with $S$ the image of a section.
We have 
$$
c_2(V(-k))=c_2(V)+(r-1)c_1(V)\cdot (-kS) + \frac{r(r-1)}{2} k^2S^2\, .
$$
Since $p_*c_1^2(V(-k))=0$ we get 
$$
p_*c_1^2(V)=2\, r \, p_*(c_1(V) \cdot kS)-r^2 \, p_*(k^2S^2)\, .
$$ 
The formula follows. The effectivity follows by Proposition
\ref{detvsMaroni} 
\end{proof}
\begin{remark}
The expression $c_2(V)-((r-1)/2r) \, c_1^2(V)$ that appears
in formula (1) is the same one 
as in Bogomolov's result on the Chern numbers 
of a semi-stable vector bundle on a surface, cf.\ \cite[Thm 0.3]{Gieseker}
\end{remark}
\end{section}
\begin{section}{Extension to Generic ${\PP}^1$-fibrations}\label{extension}
Since we want to extend the calculation of the class of a 
Maroni divisor to the full space $\overline{\mathcal H}_{d,g}$ 
of admissible covers, we have to deal with fibre spaces
which are generically a ${\PP}^1$-fibration, but over a 
divisor have singular fibres which are chains of smooth 
rational curves $P_0,P_1,\ldots, P_n$.
We call these {\sl admissible} generic ${\PP}^1$-fibrations.
We shall see later where such fibrations result from.

For defining the Maroni locus we will have to work over a base space
which is the compactified
Hurwitz space, while for most of the later calculations 
it will suffice to deal with the case of a 
$1$-dimensional complete base curve $B$ and a generic ${\PP}^1$-fibration. 
Anyway, to see what is going on, 
it might help the reader to assume that we are dealing 
with the case where the base $B$
is $1$-dimensional; in such a case we are dealing with surfaces and surface
singularities and in a singular fibre the two extremal curves
$P_0$ and $P_n$  of the chain have self-intersection number 
$-1$, while the remaining $P_i$ have $P_i^2=-2$. 
Our general case is locally a product of such a possibly 
singular surface times smooth affine space.

Let now $p:\Pi\to B$ be an admissible generic ${\PP}^1$-fibration 
over a base curve $B$ and $V$ be a vector bundle with trivial fibre 
${\mathcal O}_{{\PP}^1}^r$ on the generic fibre ${\PP}^1$ of $p$.
For such a vector bundle $V$ over $\Pi$ we can consider the  
global-to-local map
$$
p^*p_{*}V \longrightarrow V \, .
$$
We know that on the open set $U$ over which $p$ is a ${\PP}^1$-fibration,
the support of the degeneracy locus $Q$ outside a codimension $\geq 3$
locus is given by the vanishing of the determinant.
We want to analyse $Q$ near special fibres.

Even if we do not know the behavior of $V$ on the 
singular fibres of the generic ${\PP}^1$-fibration, 
we can estimate the first Chern class of  $Q$
in the following way.

\smallskip
Let $V$ be a vector bundle of rank $r$ on a generic ${\PP}^1$-fibration
$p: {\Pi}\to B$ over a smooth base $B$. 
We assume that $V$ is trivial of rank~$r$ on the generic fibre. 
The first Chern
class of $V$ restricted to a smooth fibre is trivial since it is of degree
$0$ on ${\PP}^1$, hence we can write 
$$
c_1(V)=p^* D + A,  \eqno(2)
$$
with $D$ a divisor class on $B$ and 
$A$ a divisor class supported on the singular fibres of $p$.
(We denote the divisor and its class by the same symbol.) 
Again we have a canonical morphism
$$
{\rm ev}: p^*p_* V \longrightarrow V\, .
$$
The analogue of the formula (1) in Lemma 
\ref{firstresult} is the following.

\begin{proposition}
Let $V$ be a vector bundle of rank $r$ on the admissible generic ${\PP}^1$-fibration
$\Pi$ over $B$ which is trivial on the generic fibre and with 
$c_1(V)= p^*D+A$ for $D$ a divisor class on $B$ and $A$ a divisor class
supported on the singular fibres. The first Chern class of the degeneracy locus
$Q$ of the canonical global-to-local map $p^*p_*V \to V$ is given by
the formula
$$
c_1(Q)+p^*c_1(R^1p_*V)=
A+p^*p_*\left( \frac{1}{2}A\cdot \theta_p -\frac{1}{2\, r}A^2 \right)+
p^*p_*(c_2(V)) \, .
$$
\end{proposition}  

\begin{remark}\label{independenceofA}
Note that the expression 
$A+p^*p_*((1/2)\, A\cdot \theta_p -(1/2r)\, A^2)$ 
does not change if we replace $A$ by $A+p^*D$ 
with $D$ a divisor (class) on $B$. Indeed, 
for any divisor $D$ on $B$ one has
$$
p_* (p^*D\cdot \theta_p) =-2\, D \, , \quad
A\cdot p^*D=0\quad {\rm and} \quad (p^*D)^2=0\, . 
$$
\end{remark}

\begin{proof}
We apply Grothendieck-Riemann-Roch to $p$ and $V$:
$$
\begin{aligned}
c_1(p_*V-R^1p_*V)=& p_*({\rm ch}(V) \cdot {\rm td}(T_p))_{(2)} \\
= & p_*[ (r+ c_1(V)+\frac{1}{2}(c_1^2(V)-2\, c_2(V)))(1-\frac{1}{2} \theta_p
+{\rm td}_2(T_p) )]_{(2)} \\
= &  p_*(r \, {\rm td}_2 (T_p)-\frac{1}{2}c_1(V)\cdot \theta_p +
\frac{1}{2} c_1^2(V)-c_2(V))\,  . \\
\end{aligned}
$$
Using the identities in Remark \ref{independenceofA} one sees that
$$
p_*c_1^2(V)=p_*(A^2) \quad {\rm and} \quad p_*(c_1(V)\cdot \theta_p)=
-2\, D+p_*(A\cdot \theta_p) \, .
$$ 
We claim that $p_*({\rm td}_2(T_p))=0$. For this one applies 
Grothendieck-Riemann-Roch to $p$ and ${\mathcal O}_{\Pi}$. Note that
$p_*{\mathcal O}_{\Pi}={\mathcal O}_Y$ and $R^1p_* {\mathcal O}_{\Pi}=
(p_*\omega_p)^{\vee}=(0)$, so we get
$0=c_1({\mathcal O}_Y)=p_*({\rm td}_2 (T_p))$. 
\end{proof}

Just as above, if we start with a vector bundle $V$ that is balanced 
on the generic fibre and if there exists a line bundle $M$ on $\Pi$ 
such that $V^{\prime}=V \otimes M$ is trivial of rank $r$ on the
generic fibre, then  in the formula the second Chern class gets adapted:

\begin{corollary}\label{adapted}
Let $V$ be a vector  bundle on $\Pi$ 
that is balanced on the generic fibre of $\Pi$ over $B$ and  $M$ 
a line bundle on $\Pi$ such that $V'=V\otimes M$ is trivial of rank $r$ 
on the generic fibre of $\Pi$ over $B$. We denote the degeneracy locus of
the natural map $p^*p_* V' \to V'$ by $Q$. If we write as before 
$c_1(V')=p^*D+A$ with $A$ supported on the singular fibres, 
we have
$$
\begin{aligned}
c_1(Q)+p^*c_1(R^1p_*V')=
A+p^*p_* \left( \frac{1}{2} A\cdot \theta_p -\frac{1}{2\, r}A^2 \right) 
+ \qquad &\\
p^*p_*\left( c_2(V)-  \frac{r-1}{2 \, r} c_1^2(V) \right)&  \, . \\
\end{aligned}
$$
\end{corollary}

Now we consider the effectivity of the classes occurring here.

\begin{proposition}\label{firsteffectivity}
The expression $c_1(Q)+p^*c_1(R^1p_*V')$ is an effective divisor class.
\end{proposition}
\begin{proof}
We have that $c_1(Q)$ is effective as the class of a degeneracy locus.
The stalk of the sheaf $R^1p_*V'$ over a general point is trivial as 
$V'$ is trivial over the fibre over that point. Therefore $R^1p_*V'$
is a torsion sheaf. The determinant of a torsion sheaf admits a non-trivial
regular section, see \cite[Prop.\ 5.6.14]{Kobayashi}, or alternatively,
apply the Grothendieck-Riemann-Roch theorem to the embedding of the
support of $R^1p_*V'$ and the sheaf $R^1p_*V'$ and get 
that $c_1(R^1p_*V')$
is represented by a positive multiple (the rank) 
of the codimension $1$-cycle of the support. 
Therefore
the class $c_1(R^1p_*V')$ is an effective class.
\end{proof}

We conclude two things from the above discussion:
\smallskip

\begin{enumerate}
\item{}
The class $c_1(Q)+p^*c_1(R^1p_*V')-A$ is a pullback class under $p^*$;
\item{}
The class
$c_1(Q)+p^*c_1(R^1p_*V')$
is an effective class.
\end{enumerate}

The first follows from 
Corollary \ref{adapted}, 
and the second by Proposition \ref{firsteffectivity}.
But we want a class that is both effective and a pullback under $p$.
We will do this by adding (or subtracting)  
a pullback to $c_1(Q)+p^*c_1(R^1p_*V')-A$ so that it becomes effective
in a minimal way.
We shall use the following lemma.

\begin{lemma}\label{aux}
Suppose that $D$ is an effective divisor on $\Pi$ which is linearly
equivalent to a pullback $p^*T$ from the $1$-dimensional base $B$.
If $F_0=\sum_{i=0}^n P_i$ is a reduced
fibre and if $D$ contains one irreducible
component $P_i$ with multiplicity $\eta$
then it contains all irreducible components
$P_i$ with multiplicity $\eta$.
\end{lemma}
\begin{proof}
Observe that $D$ cannot contain horizontal components 
since this would violate $D\cdot F=p^*T\cdot F =0$
for a fibre $F$. But the only solutions of $\sum_j \alpha_j P_j \cdot P_i=0$
for all $i=0,\ldots,n$ are the multiples of the fibre $F_0$.
\end{proof}

We now consider the situation over a possibly higher dimensional base.

\begin{definition}\label{notationsh}
Suppose that $G$ is a divisor on $\Pi$ supported in the
singular fibres of $p$ and suppose that its image $p(G)$ is 
supported on an irreducible divisor $D$ with  
$p^{-1}(D)=\sum_{i=0}^n E_i$ with $E_i$ irreducible. 
We assume that in a fibre over a general point $x$ of $D$ 
the components $E_0, \ldots,E_n$ give rise
to the components $P_0,\ldots ,P_n$ of $p^{-1}(x)$ forming
 a chain of rational curves.
Write $G=\sum_{i=0}^{n} e_i E_i$.
We then set
$$
G_{{\rm sh}}:=\sum_{i=0}^{n} (e_i-e_{\rm max}) E_i \, ,
$$
where $e_{\rm max}:= \max \{ e_i: i=0,\ldots,n\}$. By doing this for 
all divisors $G$ (not necessarily with irreducible image under $p$) 
 and extending it linearly we associate to any divisor $G$ 
with support in the singular
fibres a (shifted) divisor~$G_{\rm sh}$.
\end{definition}

Note that the effect of this is shifting the coefficients $e_i$ 
by a common constant such that the maximum coefficient becomes $0$.

\begin{definition}\label{notationF}
For a divisor $G$ as in Definition \ref{notationsh}  we let $F_G$ be the divisor
that is a sum of fibre divisors, each fibre divisor
$F=\sum_{i=1}^n E_i$ coming with
multiplicity $\min \{ {\rm ord}_{E_i} (G): i=1, \ldots,n \}$.
We have
$F_G= G+(-G)_{\rm sh}$.
\end{definition}

Now we improve our divisor by adding $F_A= (-A)_{\rm sh} -(-A)$.

\begin{proposition}\label{choiceA}
If $c_1(V)$ modulo a pull back of a divisor under $p$ is linearly
equivalent to a divisor $A$ supported on the singular fibres of $\Pi$ 
such that $A=\sum_D A^D$ with $D$ running through irreducible divisors
$D$ on the base $B$ such that $p^{-1}D$ is a chain $\sum_{i=1}^{n_D} E_i^D$ 
and $A^D=\sum_i e_i^D E_i^D$ then the expression
$$
c_1(Q)+p^*c_1(R^1p_*V')+(-A)_{\rm sh}
$$
is linearly equivalent to the pull back under $p$ of an {\rm effective}
class on $B$.
\end{proposition}
\begin{proof} We set $\Gamma=c_1(Q)+p^*c_1(R^1p_*V')-A$.
We may consider one irreducible divisor $D$ on the base 
$B$ such that $p^{-1}(D)=\sum_{i=0}^{n} E_i$
is  as in Definition \ref{notationsh}. If
$-A=\sum e_iE_i$ and all $e_i\geq 0$, then all numbers 
${\rm ord}_{E_i}\Gamma$ are non-negative 
and by Lemma \ref{aux} 
the divisor $\Gamma$ contains the whole fibre
$F=\sum_{i=1}^n E_i$ with multiplicity 
$\geq e_{\max}=\max \{ e_i: i=1,\ldots,n\}$
and we can then subtract it from $\Gamma$ 
while still keeping an effective divisor class which is a pull back.

Suppose then that $-A=\sum_{i=1}^{n}e_i E_i$ is not effective. 
We can replace $-A$ modulo pull backs from the base by 
the effective class
$$
-A'= -A+ |e_{\min}| F
$$
with $e_{\min}=\min \{ e_i : i=1,\ldots,n\}$. Then using Lemma \ref{aux}
we can subtract from $c_1(Q)+p^*c_1(p_*V')-A'$ a multiple $\epsilon \, F$ of $F$
with $\epsilon=\max\{e_i+|e_{\min}|: i=1,\ldots,n\}$.
So in total with $-A=\sum_{i=1}^n e_iE_i$ we can replace $-A$ in
$c_1(Q)+p^*c_1(p_*V')-A$ by 
$$
-A+|e_{\min}|F-(e_{\max}+|e_{\min}|) F= 
\sum_{i=1}^n (e_i-e_{\max}) E_i
$$ 
in order to retain an effective divisor $c_1(Q)+p^*c_1(R^1p_*V')+(-A)_{\rm sh}$. 
Doing this for all divisors $D$ with support on the singular fibres
 we replace $-A$ by $(-A)_{\rm sh}$.
\end{proof}

\begin{remark}\label{improve} 
Note that if all the coefficients of $-A=\sum_{i=1}^n e_i \, E_i$ 
are strictly negative, 
then we add something in order 
to get an effective divisor. On the other hand if $-A$ has a non-negative 
coefficient then not only  is $c_1(Q)-p^*c_1(R^1p_*V')-A$ an effective divisor (cf.\ 
Lemma \ref{aux}), 
but we can improve it, that is, find a smaller effective divisor.
\end{remark}

Using Remark \ref{independenceofA} we find the following result.

\begin{theorem}\label{firstformula}
Suppose that $V$ is a vector bundle on $\Pi$ of rank $r$ that is balanced on the 
generic fibre of $p$. 
With $A$ as in Corollary \ref{adapted} 
we find that
$$
F_A+ p^*p_*\left( \frac{1}{2}A\cdot \theta_p -\frac{1}{2\, r}A^2\right) +
p^*p_*\left( c_2(V)-  \frac{r-1}{2 \, r} c_1^2(V)\right) \, , 
$$
where
$\theta_p$ is the first
Chern class of the relative dualizing sheaf $\omega_p$ of $p$,
is represented by an effective divisor class which is pullback under $p^*$.
\end{theorem}

\end{section}
\begin{section}{A good model}\label{goodmodel}
Let ${\H}_{d,g}$ be the Hurwitz space of covers $C \to {\PP}^1$ of 
degree $d$ and genus $g$. It can be extended to the space 
$\overline{\H}_{d,g}$ of admissible covers of degree $d$ and 
genus $g$. The `boundary' $\overline{\H}_{d,g}-{\H}_{d,g}$
is a sum of finitely many divisors $S_{j,\mu}=S_{b-j,\mu}$ with $b=2g-2+2d$ 
indexed by  $2\leq j \leq b-2$ and a partition $\mu=(m_1,\ldots,m_n)$ 
of~$d$. In general the divisor $S_{j,\mu}$ will be reducible;
a generic point of a component corresponds to an admissible
cover $C \to P$ with $P$ a curve of genus $0$ having two components $P_1$
and $P_2$ intersecting in one point $Q$ such that $P_1$ (resp.\ $P_2$) has $j_1=j$ 
or  $j_1=b-j$ (resp.\ $j_2$) branch points with $j_1+j_2=b$
and the inverse image of $Q$ consists of $n$
points $Q_1,\ldots,Q_n$ with ramification indices $m_1,\ldots,m_n$.

This space $\overline{\H}_{d,g}$ is not normal, and therefore we 
consider the normalization 
${\tH}_{d,g}$ of $\overline{\H}_{d,g}$. 
This is now a smooth stack. Over ${\tH}_{d,g}$ we then have a 
universal curve $\varpi:{\tC}\to {\tH}_{d,g}$ in the sense of stacks.

Our goal is to extend the vector bundle on ${\PP}^1_{{\H}_{d,g}}$
to a vector bundle on a compactification.
For this we must extend our universal $d$-gonal cover 
and our first aim here is to construct a good model for the 
universal $d$-gonal map
$$
\gamma : {\mathcal C} \to {\PP}^1_{{\H}_{d,g}} \, .
$$ 
By a good model we mean a  proper flat map 
$\tilde{\gamma}: Y \to {\tP}$  that extends $\gamma$
over $\widetilde{\mathcal H}_{d,g}$. 

Recall that the universal curve $\tC$ over ${\tH}_{d,g}$ 
fits into a commutative
diagram
$$
\begin{xy}
\xymatrix{
{\tC} \ar[r]^c \ar[d]_{\varpi}& {\bM}_{0,b+1}\ar[d]^{\pi_{b+1}} \\
{\tH}_{d,g} \ar[r]^{h} & {\bM}_{0,b} \\
}
\end{xy}
$$
where ${\M}_{0,b}$ is the moduli space of stable $b$-pointed curves of genus $0$ 
and $\pi_{b+1}$ is the map that forgets the $(b+1)$st point.
We let ${\PP}$ be the fibre product of ${\bM}_{0,b+1}$ and ${\tH}_{d,g}$ over
${\bM}_{0,b}$. We find a diagram
$$
\begin{xy}
\xymatrix{
{\tC} \ar[r]^{\alpha} \ar[rd]_{\varpi}& {\PP}\ar[d]^{\varpi'} 
\ar[r]^{c'} &{\bM}_{0,b+1}\ar[d]^{\pi_{b+1}} \\
& {\tH}_{d,g} \ar[r]^{h} & {\bM}_{0,b} \\
}
\end{xy}
$$
At this point we consider  ${\tH}_{d,g}$ as our base $B$.
For later use we point out that normalization commutes with smooth base change.
This gives a composition of morphisms
${\tC}_{B}\to {\PP}_B \to B$.
\begin{lemma}
The spaces ${\tC}_B$ and ${\PP}_B$ both have at most $A_k$ singularities.
\end{lemma}
\begin{proof}
For the proof we may consider a $1$-dimensional base $B$.
Locally at a non-smooth point of $\varpi$ the map $\pi_{b+1}$ is given by $t=uv$,
while the map $\varpi$ is given by $x_{\nu}y_{\nu}=t_{\nu}$ and $c$ by 
$u=x_{\nu}^{m_{\nu}}$ and $v=y_{\nu}^{m_{\nu}}$ in suitable local coordinates
before we normalize the Hurwitz space. That normalization has the effect
of replacing $t$ by $s^{m}$ with $m={\rm lcm}(m_1,\ldots,m_n)$.
After pulling back $\pi$ to ${\tH}_{d,g}$ we have local equations
$s^{m}=uv$, $x_{\nu}y_{\nu}=s^{m/m_{\nu}}$ and still $u=x_{\nu}^{m_{\nu}}$ and 
$v=y_{\nu}^{m_{\nu}}$. So the local equations of ${\PP}$ over ${\tH}_{d,g}$ 
are $s^{m}=uv$ and this creates $A_{m-1}$ singularities.
In turn for $\tC$ we then find local equations at the nodes of the form
$x_{\nu}y_{\nu}=s^{m/m_{\nu}}$ which are singularities of type $A_{m/m_{\nu}-1}$.
\end{proof}

We now will work over $B$ and will suppress the index $B$.
Since ${\PP}$ is not smooth we resolve its rational singularities
and find a model ${p}: \widetilde{\PP}\to {\PP}$ resolving the singularities
in a minimal way
and take the fibre product of ${\tP}$ and $\tC$ over ${\PP}$: 
$$
Y:= \text{normalization of ${\tC}\times_{\PP} \widetilde{\PP}$} \quad
\text{and} \quad
\widetilde{Y}:=\text{ the resolution of $Y$} \, .
$$
This fits into the following commutative diagram

\begin{diagram}\label{basicdiagram}
$$
\begin{xy}
\xymatrix{
\widetilde{Y} \ar@/^/[drr]^{\rho}  \ar[rd]^{\nu} \ar[rdd]_{\tilde{\pi}} \\
& Y \ar[r] \ar[d]^{\pi} & {\tC} \ar[d] \ar[rdd]^{\varpi} \\
& {\tP} \ar[r] \ar[rrd]_p & {\PP} \ar[rd] \\
&&& B
}
\end{xy}
$$
We denote by $q: \widetilde{Y} \to B$ the composition
from upper left to lower right.
\end{diagram}

\begin{proposition}
The map $\pi: Y \to {\tP}$ is a flat map.
\end{proposition}
\begin{proof}
We use the fact that if $f: A\to B$ is a finite morphism with $A$ CohenMacaulay
(actually proven in Lemma \ref{normalizationlemma} later on)
and $B$ smooth then $f$ is flat, see \cite[6.1.5]{EGA}. 
Note that the map $\tC \to {\PP}$ is finite, hence the normalization of 
its base change too.
\end{proof}

The fact that $\pi$ is a flat map enables us to extend $V$ as a vector bundle.
However, as we shall see it will suffice to work with reflexive sheaves on smooth spaces.

Now we look at the direct image sheaves $R^iq_* {\mathcal O}_{\widetilde{Y}}$. 
By a spectral sequence argument one gets the following lemma, the proof 
of which is left to the reader.

\begin{lemma}
Suppose we have morphisms 
$X {\buildrel f \over \longrightarrow}
Y {\buildrel g \over \longrightarrow} Z$ and ${\mathcal F}$ a coherent sheaf
on $X$ with $R^jf_*{\mathcal F}=(0)$ for all $j\geq 1$. Then we have 
$R^i(g\circ f)_* {\mathcal F}= R^ig_*(f_* {\mathcal F})$.
\end{lemma}

We apply this first to $\tilde{\pi}=\pi \circ \nu$ and 
${\mathcal F}={\mathcal O}_{\widetilde{Y}}$ which has 
$R^i\nu_*{\mathcal O}_{\widetilde{Y}}=(0)$ for $i\geq 1$, 
since we resolved rational singularities only,
so that $\tilde{\pi}_*{\mathcal O}_{\widetilde{Y}}=\pi_*{\mathcal O}_Y$
and then observe that
$$
R^i\tilde{\pi}_*{\mathcal O}_{\widetilde{Y}}=
R^i\pi_*(\nu_* {\mathcal O}_{\widetilde{Y}})=
R^i\pi_*{\mathcal O}_Y=(0)
\qquad \text{for} \qquad i\geq 1
$$
by the finiteness of $\pi$.
Then we apply the same argument again to the map 
$q=p\circ \tilde{\pi}$ and we find:

\begin{lemma}\label{lemma45}
We have $R^iq_*{\mathcal O}_{\widetilde{Y}}= R^ip_*({\pi}_*{\mathcal O}_Y)$ 
for $i=0$ and $i=1$. Moreover, we have 
$p_*(\pi_* {\mathcal O}_Y)={\mathcal O}_B$.
\end{lemma}
\begin{proof}
The first statement follows directly from the above. 
Furthermore, we have 
$p_*(\pi_* {\mathcal O}_Y)=q_* {\mathcal O}_{\widetilde{Y}} 
=\varpi_* (\rho_* {\mathcal O}_{\widetilde{Y}}) ={\mathcal O}_B$.
\end{proof}
\begin{corollary}
The restriction of $\pi_*{\mathcal O}_Y$ to the generic fiber of $p$ 
is equal to ${\mathcal O}\oplus {\mathcal O}(-a_1)\oplus \cdots 
\oplus {\mathcal O}(-a_{d-1})$ with
$a_i \geq 1$ satisfying $\sum_{i=1}^{d-1} a_i= g+d-1$.
\end{corollary} 
\begin{proof} 
By Riemann-Roch the restriction of $\pi_*{\mathcal O}_Y$ to a fibre of $p$ 
has degree $-(g+d-1)$. This implies that the restriction to the generic
fibre $P$, which is a ${\PP}^1$, is $\oplus_{i=0}^{d-1} {\mathcal O}(b_i)$ with 
$b_i \geq b_{i+1}$ for $i=0,\ldots,d-2$ and $\sum b_i=-g-d+1$.
Since $q_*{\mathcal O}_{\widetilde{Y}}$ equals ${\mathcal O}_B$ we find that $b_0=0$ 
and $b_i<0$ for $i\geq 1$. We put $a_i=-b_i$ and get the result.
\end{proof}

\end{section}
\begin{section}{Local description of the map $\tilde{\pi}: \widetilde{Y} 
\to \widetilde{\PP}$}\label{Localpi}
We need an explicit description of the singularities of $Y$ and their resolution.
We consider the central part of the diagram \ref{basicdiagram} 
$$
\begin{xy}
\xymatrix{
{Y} \ar[r]^{\alpha} \ar[d]_{\pi}& {\tC}\ar[d]^{c} \\
\widetilde{\PP} \ar[r]^{\beta} & {\PP}\\
}
\end{xy}
$$
over our base $B$. We analyze the situation near a point
$s$ of $B$ that is a general point of an irreducible component
of the boundary divisor $S_{j,\mu}$.
For the description of the situation we may restrict to a $1$-dimensional
base $B$ and deal with a surface over $B$. The general situation is locally
isomorphic to the product of  such a surface times affine space.

In the following the indices $j$ and $\mu$ will be fixed and therefore
dropped from the notation. For a partition $\mu$ of $d$ we use the notation
$$
\mu =(m_1,\ldots,m_n) \qquad \text{\rm and} \quad
m=m(\mu):={\rm lcm}(m_1,\ldots,m_n)\, .
$$
Over our point $s$ the space ${\PP}$ has a singularity $\tau$ which is a node, locally
isomorphic to a quotient singularity ${\CC}^2/({\ZZ/m\ZZ})$ 
with action $(z_1,z_2)
\mapsto (\zeta_m z_1, \zeta_m^{-1}z_2)$ with $\zeta_m$ a primitive $m$th root
of unity. Analytically it is isomorphic to 
$$
A_{\tau}:={\rm Spec}({\CC}[u,v,s]/(s^m-uv))\, .
$$ 
The cover $\tC$ of ${\PP}$
has $n$ points $Q_{\nu}$ ($\nu=1,\ldots,n$) lying over our node $\tau$
with $\tC$ at the node $Q_{\nu}$ analytically given by 
the ring 
${\CC}[x_{\nu},y_{\nu},s]/(s^{m/m_{\nu}}-x_{\nu}y_{\nu})$
with the map locally given by
$$
u=x_{\nu}^{m_\nu}, \, v=y_{\nu}^{m_\nu}, \, s=s \, .
$$
We consider the resolution $\widetilde{\PP}$ of ${\PP}$. It is obtained by gluing
$m$ copies $Z_i$ ($i=0,\ldots,m-1$) 
of ${\CC}^2$ with coordinates $(\xi_i,\eta_i)$ via
$$
\xi_{i+1}=\eta_i^{-1}, \, \eta_{i+1}=\xi_i\eta_i^2,
$$
wherever this makes sense. Moreover, the map $Z_i \to {\PP}$ is locally 
given by
$$
u=\xi_i^{i+1}\eta_i^i, \, v=\xi_i^{m-i-1}\eta_i^{m-i},\, s=\xi_i\eta_i\, .
$$
The exceptional divisor is a chain of $m-1$ smooth rational curves 
$E_1,\ldots,E_{m-1}$ with $E_i$ given by the equation $\xi_{i-1}=0$ in $Z_{i-1}$ for $i=1,\ldots,m-1$, 
or equivalently by $\eta_{i}=0$ in $Z_{i}$.
The proper transform of the two ${\PP}^1$'s given by $\eta_0=0$ (resp.\
by $\xi_{m-1}=0$) lies in $Z_0$ and intersects $E_1$ (resp.\ lies in $Z_{m-1}$
and intersects $E_{m-1}$) transversally.

Let $\tilde{A}_{\tau}={\CC}[\xi_i,\eta_i]$ be the coordinate ring 
of $Z_i$. Note that the inclusion $A_{\tau}\subset \tilde{A}_{\tau}$ 
corresponding locally to the map $\beta$ is given by 
$
u=\xi_i^{i+1}\eta_i^i, \, v=\xi_i^{m-i-1}\eta_i^{m-i},\, s=\xi_i\eta_i \, .
$
Locally analytically $Y$ is given by the normalization 
of the ring $\tilde{A}_{\tau}={\CC}[\xi_i,\eta_i]$
in the quotient field $L$ of the ring 
${\CC}[x_{\nu},y_{\nu},s]/(s^{m/m_{\nu}}-x_{\nu}y_{\nu})$.
This latter ring 
locally analytically describes 
$\widetilde{\mathcal C}$ near our point and this field $L$
is given by ${\CC}(x_{\nu},s)$ and we have $u=x_{\nu}^{m_{\nu}}$.
The inclusion $\tilde{A}_{\tau} \hookrightarrow L$ is given by
$\xi_i=x_{\nu}^{m_{\nu}}/s^i$,  $\eta_i=s^{i+1}/x_{\nu}^{m_{\nu}}$.
The normalization that we want is given by the
following lemma.

\begin{lemma}\label{normalizationlemma}
The normalization is locally given by the normalization of the
coordinate ring $R={\CC}[\xi_i,\eta_i,x_{\nu}]/(x_{\nu}^{m_{\nu}}-\xi_i^{i+1}\eta_i^i)$\, .
\end{lemma}
\begin{proof}
Let $N$ be the normalization of $A_{\tau}$ in $L$.
The ring $R$ is embedded (in the same way as $\tilde{A}_{\tau}$) 
in the field $L$ and contains  $A_{\tau}$.
Observe that $s=\xi_i\eta_i \in A_{\tau}\subseteq R$ and 
$x_{\nu} \in R$ and therefore $L$ is the quotient field of $R$.
Since $x_{\nu}^{m_{\nu}}=\xi_i^{i+1}\eta_i^i$, 
we have that $x_{\nu}\in N$. Therefore
$A_{\tau}\subseteq R\subseteq N$ and $N$
and $A_{\tau}$ and $R$ have the same normalization.
\end{proof}

The surface with coordinate ring $R$ has one singularity which is 
a quotient singularity of type $(n_i,q_i)$, that is, isomorphic to
the quotient of ${\CC}^2$ by the action $(z_1,z_2) \mapsto
(\zeta_{n_i}z_1, \zeta^{q_i}_{n_i}z_2)$ with $\zeta_{n_i}$ a
primitive $n_i$th root of unity.

To see which singularity this gives we replace the equation by
$$
x_{\nu}^{n_i}=\xi_i^{\alpha_i}\eta_i^{\beta_i}
$$
with
$$
n_i=\frac{m_{\nu}}{{\rm gcd}(m_{\nu},i(i+1))}, \,
\alpha_i=\frac{i+1}{{\rm gcd}(m_{\nu},i+1)},\,
\beta_i=\frac{i}{{\rm gcd}({m_{\nu},i)}}\, .
$$
and find a quotient singularity of type 
$(n_i,q_i)$ with $q_i=-\beta_i/\alpha_i
\in ({\ZZ}/n_i{\ZZ})^*$.
\bigskip

Now we return to the general case where $B$ is the normalized
Hurwitz space $\widetilde{\mathcal H}_{d,g}$.
Applying what we found above in the case at hand, with $(j,\mu)$ fixed, 
results on $\widetilde{\PP}$ in a chain of exceptional divisors
$$
E_1,\ldots, E_{m-1} \, .
$$
Above $E_i$ we find $n$ divisors 
$T_{1,i},\ldots,T_{n,i}$ with $T_{\nu,i}$ such that when we restrict to
a fibre over general $x$ we find that  $T_{\nu,i}$ gives rise to a curve
$T^{(x)}_{\nu,i}$ that 
maps to the exceptional curve $E_i^{(x)}$ 
with degree 
$$
d_{\nu,i}:={\rm gcd}(m_{\nu},i)
$$ 
and ramification degree $m_{\nu}/{\rm gcd}(m_{\nu},i)$.

\begin{corollary}\label{branchformula}
The branch divisor $W$ of the map $\tilde{\pi}: \tilde{Y} \to \widetilde{\PP}$ 
(over $\widetilde{\mathcal H}_{d,g}$) consists of two disjoint parts:
the sum $W_{S}=\Xi$ 
of the $b=2g-2+2d$ sections 
$\Xi_i$ ($i=1,\ldots,b$) of $p$ and a contribution $W_E$
from the exceptional divisors of the map $\widetilde{Y}\to Y$ given by
$$
W_E:= \sum_{i=1}^{m-1} 
(\sum_{\nu=1}^n (m_{\nu}-d_{\nu,i}) )\,  E_i \, .
$$
\end{corollary}
\end{section}
\begin{section}{Extending our Vector Bundle}\label{extendingV}

We want to extend the vector bundle $V$ that is the dual of the kernel of the trace map 
${\rm tr}: \gamma_{*}{\mathcal O}_{\mathcal C} \to {\mathcal O}_{{\PP}^1_{{\mathcal H}_{d,g}}}$
to a vector bundle on the normalization of the compactified Hurwitz space
$\widetilde{\mathcal{H}}_{d,g}$. We shall assume that $d-1$ divides $g$ and set
$$
g=(d-1)k \, .
$$
The general theory (see \cite[Thm.\ 1.15]{Patel})
will then tell us that the locally free sheaf $V$ will be balanced on
the general fibre $P$ of $p$:
$$
V_{|P} \cong {\mathcal O}_P(k+1)^{d-1} \, .
$$

Note that $\tilde{\pi}: \tilde{Y} \to \widetilde{\PP}$ generically is a 
degree $d$ cover of smooth varieties. 
Then we take a line bundle ${\mathcal L}$ on $\widetilde{Y}$ that is trivial
when restricted to the `interior' $q^{-1}({\mathcal H}_{d,g})$:
$$
{\mathcal L}={\mathcal O}_{\widetilde{Y}}(Z),
$$
with $Z$ an effective divisor supported on the boundary. 
Since $Z$ is effective we have an inclusion
${\mathcal O}_{\widetilde{Y}} \subset {\mathcal O}_{\widetilde{Y}}(Z)$ and we thus 
get an inclusion
$\tilde{\pi}_* {\mathcal O}_{\widetilde{Y}} \subset \tilde{\pi}_*{\mathcal O}_{\widetilde{Y}}(Z)$
and since ${\mathcal O}_{\widetilde{\PP}} \subset \tilde{\pi}_* {\mathcal O}_{\widetilde{Y}}$
we get an injective homomorphism
$$
\iota=\iota_{\mathcal L}: 
{\mathcal O}_{\widetilde{\PP}} \hookrightarrow \tilde{\pi}_* {\mathcal L}\, .
$$
The dual $K_{\mathcal L}^{\vee}$ 
of the cokernel $K_{\mathcal L}$ 
of $\iota$ is a reflexive sheaf since it is the dual of a coherent sheaf and since we
are working on smooth spaces by neglecting an
algebraic subset of codimension $\geq 3$, 
we may and will assume that it is locally free
(see \cite[Cor.\ 1.4]{Hartshorne}). 
Its restriction to $p^{-1}({\mathcal H}_{d,g})$ is isomorphic to the bundle $V$
considered before. We will 
denote this rank $d-1$ bundle $K_{\mathcal L}^{\vee}$ on $\widetilde{\PP}$ by 
$$
V_{\mathcal L}:=K_{\mathcal L}^{\vee}\, .
$$

The choice of $Z$ will give us freedom that we shall use later. But we
start by assuming that $Z$ is trivial, that is, we start by assuming
$$
\mathcal{L}= \mathcal{O}_{\widetilde{Y}}\, .
$$
We then find a rank $d-1$ bundle denoted by 
$V=V_{\mathcal{O}_{\widetilde{Y}}}$. In fact, 
by the results of section \ref{goodmodel} 
this special $V$ is not only a reflexive sheaf, 
but actually a vector bundle. 

\begin{lemma}\label{Wlemma}
We have 
$$
c_1(V)=-c_1(\tilde{\pi}_{*} \mathcal{O}_{\widetilde{Y}})=W/2\, .
$$
\end{lemma}
\begin{proof}
The first equality follows directly from the definition, while the second follows by
applying Grothendieck-Riemann-Roch to $\tilde{\pi}$ and $\mathcal{O}_{\widetilde{Y}}$. 
\end{proof}
Note that the vector bundle $V$ is not trivial on the generic fibre.
This issue will be addressed  now by considering a specific twist
$V'=V\otimes M$ with an appropriate line bundle $M$ 
that makes $V'$ trivial on a generic fibre. 

As described in section \ref{Localpi} we analyze the situation near a point $s$
of the base $B$ that is a general point of an irreducible component  $\Sigma$ of
the  boundary divisor $S_{j,\mu}$. As explained there it suffices to
consider the case where $B$ is $1$-dimensional.

We thus consider a point $s \in B$ that is a general point of $\Sigma$ 
and the fibre of $\widetilde{\PP}$ over it.
It is a chain
$$
P_1,E_1,\ldots,E_{m-1},P_2  \eqno(3)
$$
of smooth rational curves.
The curve $P_1$ is a ${\PP}^1$ with $j_1=j$ or $b-j$ marked branch
points and likewise $P_2$ is a  copy of ${\PP}^1$ 
with $j_2$ marked branch points with $j_1+j_2=b$, 
see the first paragraph of section \ref{goodmodel}. 

Recall that we have fixed a pair $(j,\mu)$ with $\mu=(m_{1},\ldots,m_{n})$.
Assume now that $j_1=b-j$ and $j_2=j$. 
This choice  will not affect our conclusion as we will see at the end of Section \ref{classmst}.
 We then have
$$
b-j+ d-n \equiv \, 0 \, (\bmod \, 2) \qquad \text{and} \qquad j+d-n \equiv \, 0 \, (\bmod \, 2)
$$
because these are the degrees of ramification divisor of a curve over 
$P_1$ and $P_2$; moreover, the dual of the kernel of 
$\tilde{\pi}_* {\mathcal O}_{\widetilde{Y}} \to {\mathcal O}_{\widetilde{\PP}}$,
has degrees 
$$
\text{\rm $(b-j+d-n)/2$ on $P_1$ \quad and \quad $(j+d-n)/2$ on $P_2$}\, .
$$
We divide the latter degree by $d-1$ and write $r$ for the remainder
$$
\frac{j+d-n}{2}= q(d-1)+r \qquad \text{with} \qquad  0 \leq r < d-1, \eqno(4)
$$
where of course $q$ and $r$ depend on $j$ and $d$. 

For uniformity of notation we rename the divisors in (3) by $R_i=R_i^{\Sigma}$ with
$$
R_0=P_1,\, R_1=E_1,\ldots,\, R_{m-1}=E_{m-1},\, R_{m}=P_2\, .
$$

We now twist $V$ by a line bundle $M$ such that $V \otimes M$
is trivial on the generic fibre. For this we could take as a first approximation 
the line bundle $M$ corresponding to the divisor $-(k+1) \Xi_1$
with $\Xi_1$ the first section. 

Now $V\otimes M$ is trivial on the generic fibre, hence as in equation (2)
we have
$$
c_1(V\otimes M)= p^*D+[A],
$$
with $D$ a divisor class on our base space and $A$ a divisor supported on the singular
fibre $s$. As explained in section 3 we can get an effective class of the form
$c_1(Q)+p^*c_1(R^1p_*V')+ (-A)_{\rm sh}$ and for this we wish to
minimize the expression $(-A)_{\rm sh}$. In other words, we want
to adapt $M$ such that we get a `good' $A$.

By Lemma \ref{Wlemma} we know that $c_1(V)=W/2$ and by using the description of $W$ given in 
Corollary \ref{branchformula}, we find the degrees of the restriction of $V$ on the various 
components $R_i$. This guides us to a reasonable choice of  $M$ and  $A$. 

We now add to $-(k+1)\Xi_1$  an integral linear combination 
of the curves from the chain $R_0,R_1,\ldots, R_{m-1},R_m$ that has degree
$0$ on all the exceptional curves $E_i=R_i$ ($i=1,\ldots,m-1$) 
such that the degree of $V^{\prime}$ is
transfered as much as possible to $R_0=P_1$. Note that the expression
$$
\sum_{i=0}^{m} i \, R_i \quad \text{(resp.\ $
\sum_{i=0}^{m} (m-i)\,  R_i $)}
$$
has degree $0$ on all $E_i$, degree $1$ on $P_1$ and degree $-1$ 
on $P_2$ (resp.\ degree $-1$ on $P_1$ and degree $1$ on $P_2$). 

We thus arrive at the standard line bundle $M$ by which we will twist our vector bundle.

\begin{definition}\label{standardM} (Standard $M$.) 
We let $M$ be the line bundle on $\widetilde{\PP}$ associated to the divisor
$$
-(k+1) \, \Xi_1 -  \begin{cases} \sum_{\Sigma}
q\, \sum_{i=0}^{m} (m-i) R_i^{\Sigma} & \hbox{\rm if}\quad 
\Xi_1 \cdot R_0^{\Sigma}\neq 0, \\ 
\sum_{\Sigma}  (k+1-q)  
\sum_{i=0}^{m} i\,  R_i^{\Sigma} & \hbox{\rm if} \quad \Xi_1 \cdot R_m^{\Sigma} \neq 0\, ,\\
\end{cases}
$$
where the sum is taken over all irreducible boundary divisors $\Sigma$ of
$\widetilde{\mathcal{H}}_{d,g}$, $m=m(\mu)$, and $q$ is given in (4).
\end{definition}
We summarize the result on the degrees of this vector bundle on the curves $R_i$.

\begin{lemma}\label{Mconstruction}
For given irreducible component of $\Sigma$ of $S_{j,\mu}$ 
the degrees of the vector bundle 
$V'=V\otimes M$ on the curves
$R_0^{\Sigma},\ldots,R_m^{\Sigma}$ of the chain are:
$d-n-r$ on $R_0$, $r$ on $R_m$ and $\deg W_E/2$ on $R_i$,
that is  $(-d+2n-\sum_{\nu=1}^n d_{\nu,2})/2$ on $R_1$ and $R_{m-1}$, 
while on $R_i$ for $2 \leq i \leq m-2$ the degree is
$$
\frac{1}{2} \sum_{\nu=1}^n -d_{\nu,i-1} +2d_{\nu,i}-d_{\nu,i+1}\, ,
$$ 
where as before $n=n(j,\mu)$ and $r=r(j,\mu)$ and
$d_{\nu,i}={\rm gcd}(m_{\nu},i)$.
The total degree of $V'$ on the chain is equal to $0$.
\end{lemma}

\bigskip
 
Now we want to find a divisor $A^{\Sigma}$ with support in the chain  
$R_0^{\Sigma},\ldots,R_{m}^{\Sigma},$
that satisfies for a general point $s \in \Sigma$, 
the condition
$$
\deg V' = \deg A^{\Sigma}
$$
when restricted to the fibre over $s$.

In the following discussion we keep $\Sigma$ and the pair $(j,\mu)$ fixed,
therefore we allow ourselves occasionally to drop $\Sigma$ from the notation.
We write $v_i$ for the degree of $V'$ on a general fibre of 
$R_{i}^{\Sigma}$ so that we have $\sum_{i=0}^{m}v _i=0$. 
Our sought-for $A^{\Sigma}$ will be given by
$$
-A^{\Sigma}=\sum_{i=0}^{m-1}\alpha_i \, R_i^{\Sigma}\, ,
$$ 
thus leaving out $R_{m}^{\Sigma}$. 
Starting with $\alpha_{m-1}= v_0+\ldots+v_{m-1}$ gives the
right degree on $R_m$ and solving step by step via 
$\alpha_{m-i-1}=\alpha_{m-i}+ \sum_{t=0}^{m-i-1} v_t$ 
we find the solution
with $\alpha_i$ given by
$$
\alpha_i=(m-i) (\sum_{t=0}^{m-1} v_t) -\sum_{t=i}^{m-1} (t-i) v_t \, .
$$
We can extend the formula of Lemma \ref{Mconstruction}
$$
v_i= \frac{1}{2} \sum_{\nu=1}^n -d_{\nu,i-1}+2\, d_{\nu,i} -d_{\nu,i+1}
$$
for the degree of $V'$ on a general fibre of 
$R_i=E_i$ for $i=2,\ldots,m-1$ to $i=0$ and $i=1$
by putting formally
$$
d_{\nu,0}=d_{\nu,m}=m_{\nu}, \qquad d_{\nu,-1}:= \frac{2r}{n}+1 \, . 
$$
Note that 
$$
-\sum_{t=i}^{m-1} (t-i) (-d_{\nu,t-1}+2\, d_{\nu,t}-d_{\nu,t+1})=
(m-i-1) d_{\nu,m}-(m-i)d_{\nu,m-1}+d_{\nu,i}.
$$
Then we find (using $\sum_{t=0}^{m-1} v_t=-v_{m}$) that 
$$
\begin{aligned}
\alpha_i & = -r(m-i) +\frac{1}{2} \sum_{\nu=1}^n 
( d_{\nu,i} -(m-i)\,  d_{\nu,m-1}+(m-i-1) \, d_{\nu,m} ) \\
& = -r +\frac{1}{2} \sum_{\nu=1}^n 
( d_{\nu,i} -(m-i) +(m-i-1)\, m_{\nu})\\
& = \frac{1}{2} ( (m-i)(d-n-2r)-d +  \sum_{\nu=1}^n   d_{\nu,i}) .\\
\end{aligned}
$$
\begin{definition}\label{c-deltai-def}
Now define $c=c(j,\mu)$ and $\delta_i=\delta_i(\mu)$ 
$$
c:=d-n-2r \qquad \text{\rm and}\qquad 
\delta_i:=d-\sum_{\nu=1}^n d_{\nu,i} \quad \text{\rm for $i=0,\ldots,m$}
$$ 
with $d_{\nu,i}$ defined as ${\rm gcd}(m_{\nu},i)$ for $i\geq 1$, 
by $d_{\nu,0}=m_{\nu}$ and $d_{\nu,-1}=2r/n+1$.
\end{definition}
Then we have
$\delta_0=0$ and we can write $-A=\sum_{i=0}^{m-1} \alpha_i R_i$ with
$ \alpha_i= \frac{1}{2} ((m-i)c -\delta_i)$.
Observe that 
$$
c= \deg V^{\prime}_{|P_1}-\deg V^{\prime}_{|P_2}
$$
and by  Corollary \ref{branchformula}
$$
W_E=\sum_{i=1}^{m-1} \delta_i R_i \, .
$$

\begin{conclusion}\label{Acontribution}
For each irreducible component $\Sigma$ of $S_{j,\mu}$ the divisor 
$$
A^{\Sigma}=-\sum_{i=0}^{m-1} \frac{1}{2} ((m-i)c -\delta_i) \, R_i^{\Sigma}
$$
has the property that the degree of $A^{\Sigma}$ when restricted to a general fibre
of $R_i^{\Sigma}$ equals the degree of $V'$ restricted to that fibre.
\end{conclusion}

For $\mathcal{L}$ trivial, $V=V_{\mathcal{L}}$ and with the standard choice
for $M$ as in Definition \ref{standardM}
we set
$$
A_{\rm st}:= \sum_{\Sigma} A^{\Sigma}  \eqno(5)
$$
with $A^{\Sigma}$ as in Conclusion \ref{Acontribution} .
Then the first Chern class $c_1(V')$ of $V'=V \otimes M$ satisfies the equality (2)
$$
c_1(V')=p^*D+A_{\rm st}
$$
with $D$ a divisor class on the base.
\end{section}
\begin{section}{Extensions of the Maroni Divisor}\label{thedivisor}
In this section we will work on the Hurwitz space 
$\overline{\mathcal H}_{d,g}$ and its normalization, where we assume as before that 
$g=(d-1)k$. In the preceding section we have constructed a 
vector bundle $V$ of rank $d-1$ on 
$\widetilde{\mathcal H}_{d,g}$ and a twist $V'=V\otimes M$
by an explicit line bundle $M$ 
such that $V^{\prime}$ is trivial on the generic fibre.

Now we define an effective divisor on ${\tH}_{d,g}$
by applying Proposition \ref{choiceA} and 
Theorem \ref{firstformula}  
of Section \ref{extension} to the bundle $V'$.
This involves a divisor $-A$ given by $c_1(V^{\prime})\equiv A$ modulo a pull back from the base
 and we make the result effective 
by adding $F_A$ as in Definition \ref{notationF} for the choice of A as in  Conclusion \ref{Acontribution}.

\begin{definition}\label{extendedMaroni}
(The standard extended Maroni divisor class $\mathfrak{m}_{\rm st}$.)
Let  $\mathcal{L}$ be trivial line bundle on $\widetilde{Y}$, $V=V_{\mathcal L}$ 
and choose  $M$ as in  Definition \ref{standardM} and put $V'=V\otimes M$.
If $Q$  denotes the degeneracy locus of the evaluation map $p^*p_*V' \to V'$ and
with $A=A_{\rm st}$ 
as in (5) 
the class
$$
c_1(Q)+p^*c_1(R^1p_*V')+(-A)_{\rm sh}
$$ 
is effective and
pulling it back via a section of $\widetilde{\PP}\to {\tH}_{d,g}$
defines an effective divisor class $\mathfrak{m}_{\rm st}$ on 
${\tH}_{d,g}$ 
which is the pull back under a section of $p$ of the class
$$
p^*p_*\left( \frac{1}{2}[A]\cdot \theta_p -\frac{1}{2\, r}[A]^2\right) +
p^*p_*\left( c_2(V)-  \frac{r-1}{2 \, r} c_1^2(V)\right) +F_A \, , 
$$
and agrees with the Maroni locus on the open part
${\mathcal H}_{d,g}$. We call $\mathfrak{m}_{\rm st}$  the {\sl standard extended Maroni class}.
\end{definition}
We refer to Definition \ref{notationsh} for the notation $(-A)_{\rm sh}$,
to \ref{notationF} for the notation $F_A$ 
 and Remark \ref{improve} for the meaning of $F_A$.

We can vary this definition by taking  $\mathcal{L}=\mathcal{O}_{\widetilde{Y}}(Z)$ 
associated to an effective divisor $Z$ with support in the boundary of
$\widetilde{\mathcal{C}}$
and by twisting $V_{\mathcal L}\otimes M$ by a line bundle $N$ with 
support in the boundary of $\widetilde{\PP}$.
 
\begin{definition}\label{defgeneralm} 
(The class $\mathfrak{m}_{\mathcal{L},N}$.) By taking  $\mathcal{L}$ an effective line bundle
$\mathcal{L}=\mathcal{O}_{\widetilde{Y}}(Z)$  as in Section \ref{extendingV}, 
letting $V'=V'_{\mathcal{L},N}= V_{\mathcal{L}} \otimes M
\otimes N$ and letting $Q$ be the degeneracy locus of the evaluation map
$p^*p_* V' \to V'$ and defining 
$A=A_{\mathcal{L},N}$ as in (2)  the class
$$
c_1(Q) +p^*c_1(R^1p_*V')+
(-A)_{\rm sh}
$$
is effective and
pulling it back via a section of $p: \widetilde{\PP}\to {\tH}_{d,g}$
defines an effective divisor class $\mathfrak{m}_{\mathcal{L},N}$ on 
${\tH}_{d,g}$ 
which by Theorem \ref{firstformula} 
is the pull back under a section of $p$ of the class
of the form
$$
p^*p_*\left( \frac{1}{2}A \cdot \theta_p -\frac{1}{2\, r}
A^2\right) +
p^*p_*\left( c_2(V)-  \frac{r-1}{2 \, r} c_1^2(V)\right) +F_{A} \, ,
$$
and agrees with the Maroni locus on the open part
${\mathcal H}_{d,g}$.
\end{definition}
For $\mathcal{L}$ and $N$ trivial we have
$\mathfrak{m}_{\mathcal{L},N}=\mathfrak{m}_{\rm st}$. 
In the next section we will work out the various terms in the formula
for the basic case $\mathfrak{m}_{\rm st}$ of the standard extended Maroni class.
\end{section}
\begin{section}{The Class of the Standard Extended Maroni Divisor}\label{classmst}
We now  calculate the class $\mathfrak{m}_{\rm st}$
defined in the preceding section.
We begin with the Chern classes in the formula of Definition \ref{extendedMaroni}.
Recall the diagram
$$
\begin{xy}
\xymatrix{
{Y} \ar[r]^{\alpha} \ar[d]_{\pi}& {\tC}\ar[d]^{c} \\
\widetilde{\PP} \ar[r]^{\beta} & {\PP}\\
}
\end{xy}
$$
We also recall that by Lemma \ref{Wlemma} we have $c_1(V)= W/2$.
Note that the branch divisor $W$ of $\pi$ was calculated in
Corollary \ref{branchformula} 
and consists of two disjoint parts:
the sum $W_S=\Xi$ of the $b$ sections 
and a contribution $W_E=\sum_{i=1}^{m-1} \delta_i E_i$ from the exceptional
divisors.

Now we move to the second Chern class.
We have $c_2(V)=c_2(\pi_* {\mathcal O}_Y)$, so we consider 
$q_!{\mathcal O}_{\widetilde{Y}}$.
On the one hand we have $q=\varpi \circ \rho$
and observe that
$\rho_{!}{\mathcal O}_{\widetilde{Y}}=\rho_{*}{\mathcal O}_{\widetilde{Y}}$
because $\rho$ blows down resolutions of rational singularities.
Hence we get
$$
q_{!}{\mathcal O}_{\widetilde{Y}}=
\varpi_{!}(\rho_{*}{\mathcal O}_{\widetilde{Y}})
=\varpi_{!}{\mathcal O}_{\tilde{\mathcal C}}=
 {\mathcal O}_B - {\EE}^{\vee} 
$$
with ${\EE}=\varpi_{*}\omega_\varpi$,  the Hodge bundle,
by Serre duality the dual of $R^1\varpi_{*}{\mathcal O}_{\widetilde{C}}$.
We conclude that
$$
c_1(q_{!}{\mathcal O}_{\widetilde{Y}})= \lambda\, ,
$$
where $\lambda$ denotes the Hodge class.
On the other hand we have $q=p\circ \pi \circ \nu$ (see Diagram \ref{basicdiagram}) 
and we have by Lemma \ref{lemma45} that
$q_{!}{\mathcal O}_{\widetilde{Y}}=p_{!}(\pi_*{\mathcal O}_Y)$ and this can be
calculated by applying Grothendieck-Riemann-Roch
to the morphism $p$ and the sheaf $\pi_*{\mathcal O}_Y$.
This gives
$$
\begin{aligned}
c_1(p_{!}(\pi_{*}{\mathcal O}_Y))=&
p_*[{\rm ch}(\pi_{*}{\mathcal O}_Y) \, {\rm td} T_p]_{(2)}\\
=& p_{*}\big[d\, {\rm td}_2T_p 
-\frac{1}{2}c_1(\pi_{*}{\mathcal O}_Y) \, \theta_p +
\frac{c_1^2(\pi_{*}{\mathcal O}_Y)-2\, c_2(\pi_{*}{\mathcal O}_Y)}{2} 
     \big] \\
=& p_{*}[\frac{1}{4} W \theta_p + \frac{1}{8} W^2 -c_2(V) ]\, ,\\
\end{aligned}
$$
where we used that $p_{*}({\rm td}_2T_p)=0$ and
$c_1(\pi_{*}{\mathcal O}_Y)=-W/2$, see Lemma \ref{Wlemma}.
We thus see that
$$
\lambda = c_1(q_{!}{\mathcal O}_{\widetilde{Y}})= 
p_{*}\big( W \, \theta_p /4 + W^2/8 -c_2(V)\big) 
$$
and thus obtain a formula for $p_*c_2(V)$:
$$
p_{*}c_2(V)=-\lambda + \frac{1}{8} p_{*}( W^2+2 W\, \theta_p)\, .
$$
We summarize.

\begin{proposition}\label{Chernclasses}
We have
$$
p_{*} c_1^2(V)= \frac{1}{4}p_{*} W^2 \quad \hbox{\rm and} \quad 
p_{*}c_2(V)=-\lambda + \frac{1}{8} p_{*}( W^2+2 W\, \theta_p)\, .
$$
\end{proposition}
From now on we shall work over a $1$-dimensional base
$B$ in the Hurwitz space. We need to justify that it suffices
to prove an identity between divisor classes on
$\widetilde{\mathcal{H}}_{d,g}$ by verifying it 
for all $1$-dimensional curves $B$ in $\widetilde{\mathcal{H}}_{d,g}$ 
that intersect the boundary of our 
normalized compactified Hurwitz
space $\widetilde{\mathcal{H}}_{d,g}$ transversally at general points.
By repeatedly intersecting with generic hyperplane sections and applying
the Lefschetz hyperplane theorem we reduce ourselves to a smooth
surface $S$ with the property that the Picard group of 
$\widetilde{\mathcal{H}}_{d,g}$ injects into ${\rm Pic}(S)$. 
The boundary of the Hurwitz space defines a configuration of transversally
intersecting curves $T$ on $S$ and we assume that we
have a line bundle $L$ (representing our divisorial identity) on $S$
which is trivial on every curve $B$ that intersects $T$ transversally at
sufficiently general points. We wish to conclude that $L$ is trivial.
We claim that the N\'eron-Severi group of $S$ is generated by the
classes of such curves $B$. Indeed, the general hyperplane section $H$ 
of $S$  intersects $T$ transversally at general points and 
hence $L$ is trivial on it.
For any curve $E$ the class $E+mH$ contains for sufficiently large $m$ 
a representative curve that intersects $T$ as neatly as desired. 
Hence the first Chern 
class of $L$ is orthogonal to all of ${\rm NS}(S)$, so by the
non-degenerateness of the intersection form a non-zero 
multiple of it vanishes. 
Hence a non-zero integral multiple of it comes from a class in 
$H^1(S,\mathcal{O}_S)$
under the exponential map $H^1(S,\mathcal{O}_S)\to H^1(S,\mathcal{O}_S^*)$.
But by Kodaira and Spencer the map $H^1(S,\mathcal{O}_S)\to
H^1(H,\mathcal{O}_H)$ is injective for a general hyperplane section of $S$, 
hence if $L$ restricts trivially to $H$ it must be trivial. 
Since we work in the rational Picard group this suffices.

\bigskip

If $m(\mu)>1$ (recall that $m(\mu)$ is the l.c.m. of the $m_i$)
then the self-intersection number $W_E^2$ is negative.
Indeed, since the chain $E_1,\ldots,E_{m-1}$ can be contracted its
intersection matrix is negative definite and since $W_E$ is not the zero
divisor the result follows.
We can calculate this self-intersection number as follows.
For each pair $(j,\mu)$ we have $W_E=\sum_{i=0}^m \delta_i E_i$ with
$\delta_0=\delta_m=0$ (see Definition \ref{c-deltai-def})
and
$W_E^2= 
\sum_{i=1}^{m-1} (-2 \delta_i^2 +2 \, \delta_i \delta_{i+1})$,
in other words,
$$
W_E^2= -\sum_{i=1}^{m} (\delta_{i}-\delta_{i-1})^2\, .
$$

Note that we have $W_E\cdot \theta_p=0$ and thus $W\cdot \theta_p=
\Xi\cdot \theta_p=\psi$ (with $\Xi$ given in Corollary \ref{branchformula} 
and $\psi$ the usual sum of the tautological classes
$\psi_i$ defined by the sections)
which gives for each intersection point of $B$ with $S_{j,\mu}$ a contribution
$m(\mu)\, j(b-j)/(b-1)$,
see formula (2) for $\psi$ in \cite{GK}. 
Since $p_*(W_S^2)=-\psi$ we get from each intersection point of $B$ with
$S_{j,\mu}$ a contribution
$$
-\frac{j(b-j)}{b-1} m(\mu) \, .
$$

Now we move to the contributions of the divisor $A$. 
First recall the definition of $A$. For each irreducible component $\Sigma$
of $S_{j,\mu}$ we have a contribution $-A^{\Sigma}$ with coefficients given by
$$
\alpha_i= \frac{1}{2} ((m-i)c -\delta_i)
$$
and have for the contribution of the point $s \in B \cap S_{j,\mu}$ 
to the self-intersection number of $A$
$$
\begin{aligned}
 -\alpha_0^2+2 \sum_{i=1}^{m-1} \alpha_i(\alpha_{i-1}-\alpha_i) &=
 -\frac{m^2c^2}{4} + \frac{1}{2} \sum_{i=1}^{m-1}((m-i)c-\delta_i)
(c-\delta_{i-1}+\delta_i) \\
&= -\frac{mc^2}{4}
-\frac{1}{2} \sum_{i=1}^{m-1} \delta_i(\delta_i-\delta_{i-1}) \, , \\
\end{aligned}
$$
where we used 
$\sum_i (m-i)(\delta_{i-1}-\delta_i)=-\sum_i{\delta}_i$.
Noting that $\delta_0=\delta_m=0$ and using again the identity
$$
2\sum_{i=1}^{m-1} \delta_i(\delta_i-\delta_{i-1})=
\sum_{i=1}^m (\delta_i-\delta_{i-1})^2
$$
we can rewrite this as 
$$
-\frac{mc^2}{4} -\frac{1}{4} \sum_{i=1}^m (\delta_{i-1}-\delta_i)^2\, .
$$
In general we will get a contribution to $A$ from the resolution divisor
of any intersection point $s$ of $B$ with $S_{j,\mu}$. 
This contribution depends only on the pair $(j,\mu)$, not on the
chosen point. Therefore we get a contribution
$$
-A= \sum_{j,\mu} (B\cdot S_{j,\mu}) \sum_{i=0}^{m(\mu)} \alpha_i(j,\mu) \,
R_i(j,\mu) \, .
$$
For the term $A\cdot \theta_p$ we get the negative of the coefficient of the
`main' component $R_0$ which equals
$m(\mu) \, c /2$.
We thus arrive at the following lemma.
\begin{lemma} The contribution of $S_{j,\mu}$ to the self-intersection
number of $A$ is for each intersection point $s$ of $S_{j,\mu}$ with $B$
equal to
$$
-\frac{mc^2}{4} -\frac{1}{4} \sum_{i=1}^{m} (\delta_i-\delta_{i-1})^2\, ,
$$
while the contribution to $A\cdot \theta_p$ is 
$\frac{1}{2} m(\mu) c$.
\end{lemma}

As to the term $\lambda$ occuring in our formula of the extended
Maroni class we recall the formula for the Hodge class $\lambda$
from \cite{KKZ}, which was reproved in an algebraic way in \cite{GK}. 
This formula says:
$$
\lambda= \sum_{j=2}^{b/2} \sum_{\mu} m(\mu)
\left( \frac{j (b-j)}{8(b-1)} -\frac{1}{12}(d-\sum_{\nu=1}^{n(\mu)} \frac{1}{m_{\nu}})
\right) \, S_{j,\mu} \, .
$$

We have all the ingredients for the cycle class of the extended Maroni 
divisor.

\begin{theorem}\label{classMaroni}
The class $\mathfrak{m}_{\rm st}$ of the standard extened Maroni divisor
equals the effective class 
$\sum \sigma_{j,\mu} S_{j,\mu}$ with $\sigma_{j,\mu}$ given by
$$
\sigma_{j,\mu}=
m(\mu)\, \left( -\frac{|c_{j,\mu}|}{4}+ \frac{c_{j,\mu}^2}{8(d-1)} +
\frac{1}{12}(d-\sum_{\nu=1}^{n(\mu)} \frac{1}{m_{\nu}})+
\frac{j(b-j)(d-2)}{8(b-1)(d-1)}\right) \, ,
$$
where $\mu=(m_1,\ldots,m_{n(\mu)})$ and $c_{j,\mu}=d-n-2r$ with $r$ the remainder
of $(j+d-n)/2$ by division by $d-1$.
\end{theorem}
\begin{proof}
We use the definition of \ref{extendedMaroni} 
and collect the various terms for a given
pair $(j,\mu)$: the contribution of $p^*p_*(A \cdot\theta_p/2 - A^2/2(d-1))$
yields
$$
\frac{1}{4} m(\mu)c_{j,\mu} + \frac{1}{8(d-1)}
(m(\mu) c_{j,\mu}^2
+\sum_{i=1}^{m(\mu)} (\delta_i-\delta_{i-1})^2) \, .
$$
The contribution to $F_A$ is given by the minimum coefficient: 
if $c_{j,\mu}\leq 0$ then  the  coefficients are $\geq 0$ with the minimum 
corresponding to $R_m$ which is $0$; if $c_{j,\mu}>0$ then the minimum is 
$-mc_{j,\mu}/2$.

The contribution of $-\lambda$ is
$$
-m(\mu) \frac{j(b-j)}{8(b-1)} + m(\mu) \frac{1}{12}\left( 
d-\sum_{\nu} \frac{1}{m_{\nu}} \right)
$$
and the contribution of $p_*(W\cdot \theta_p/4 + 
([W_S]^2+[W_E]^2)/8(d-1))$ is
$$
\frac{j(b-j)}{4(b-1)}\, m(\mu) -\frac{j(b-j)}{8(d-1)(b-1)} \, m(\mu)
-\frac{1}{8(d-1)} \sum_{i=1}^{m(\mu)} (\delta_{i}-\delta_{i-1})^2\, .
$$
\end{proof}

In Theorem \ref{classMaroni} the constant $c_{j,\mu}$ depends on the choice $j_2=j$, the number of
branch points on $P_2$. If we considered $j_2=b-j$ instead, the role of $j$ would be
taken by $l=b-j$. Instead of $c=d-n-2r$ we would get a constant 
$c'=d-n-2r'$ with $r'$ the remainder of $(l+d-n)/2$ by division by $d-1$. But
the formula is invariant under the change from $c$ to $c^{\prime}$ as witnessed by
the following lemma the proof of which is left to the reader.

\begin{lemma}\label{cversuscprime}
We have $|c|(|c|-2(d-1))=|c^{\prime}|(|c^{\prime}|-2(d-1))$.
\end{lemma} 
\end{section}
\begin{section}{Varying the Maroni Class by Twisting the Vector Bundle}\label{twistingV}
The standard extended Maroni divisor $\mathfrak{m}_{\rm st}$ 
was obtained by starting 
with ${\mathcal L}$ trivial and taking $V=V_{\mathcal{L}}$
and then
by tensoring with the standard $M$ given in Definition \ref{standardM} 
yielding a vector bundle $V'$ trivial on the generic fibre. 
But we can replace $V'$ by $V'_N=V' \otimes N$
with $N$ any line bundle associated to a 
divisor supported on the singular fibres
and preserve the property that the vector bundle is trivial on the
generic fibre. In particular, we thus get rid of the special choice
of the $k+1$ sections in Definition \ref{standardM}.
We have $c_1(V_N)=p^*D+ A_N$ with $A_N=A_{\rm st}+(d-1)N$.
Using Definition \ref{defgeneralm} we get an
effective class $\mathfrak{m}_{\mathcal{L},N}$ 
which is a pull back from $\widetilde{\mathcal H}_{d,g}$
and agrees with the Maroni locus on the open part
${\mathcal H}_{d,g}$.  In this section we shall drop the trivial $\mathcal{L}$ from
the notation and will write $\mathfrak{m}_{N}$ for this class.
For $N$ trivial we get back $\mathfrak{m}_{\rm st}$.

Again we may consider the situation locally  as we did before.
That is, we are working near an irreducible component $\Sigma$ of the boundary 
divisor $S_{j,\mu}$ of
the base $\widetilde{\mathcal{H}}_{d,g}$. 
But we may work on a $1$-dimensional base  $B$ 
and assume we are dealing with the multiplicity of one point $s$ on $B$
with fibre $F$ that corresponds to the transversal intersection of 
$B$ with the boundary divisor $\Sigma$ 
of $\widetilde{\mathcal{H}}_{d,g}$ at a generic point of $\Sigma$.

We thus may  write (abusing the notation slightly and omitting the index
$\Sigma$) 
$N=\sum_{i=0}^{m}a_i R_i$ with $a_m=0$ and similarly we write $A$ instead of
$A^{\Sigma}$. 

We find that the difference
$\mathfrak{m}_{\rm st}-\mathfrak{m}_{N}$
is the pull back under a section of $p$ of
$$
F_A-F_{A+(d-1)N} +
p^*p_*\left( N\cdot A+ \frac{d-1}{2} \, (N^2-N\cdot \theta_p)\right) \, ,
$$
and so
${\rm ord}_{\Sigma}\mathfrak{m}_{\rm st}-{\rm ord}_{\Sigma}\mathfrak{m}_{N}$
is the degree of the singular fibre $F$ in
$$
F_A-F_{A+(d-1)N}+
f(N)\, F \eqno(6)
$$
with $f$ the quadratic function
$$
f(N)=N\cdot A+ \frac{d-1}{2} \, (N^2-N\cdot \theta_p)\, .
$$
If for some $N$ we find that $\mathfrak{m}_{\rm st}-\mathfrak{m}_{N}$ 
is positive, we can improve the extended Maroni divisor 
by finding an effective divisor
$\mathfrak{m}_{N}$ that contains the closure of the Maroni locus,
 but is smaller
than $\mathfrak{m}_{\rm st}$.  

Discarding the term $F_{A}-F_{A+(d-1)N}$ in (6)
we look for the maximum of the function $f$. The partial derivative of $f$ 
with respect to $N$ is $A+(d-1)(N-\theta_p/2)$ and this vanishes for
$$
N_{\rm crit}=\frac{(d-1)\theta_p-2A}{2(d-1)}\, .
$$
Now $\theta_p$ near our fibre is represented by 
$-2S+\sum_{i=0}^{m-1} (m-i) R_i$,
where $S$ is a section that intersects $P_2=R_m$.
Using the form we found for $A$ in Conclusion 
\ref{Acontribution}
we see that the coordinates of $N_{\rm crit}=\sum_{i=0}^{m-1} a_i R_i$ 
are given by
$$
a_i=\frac{(m-i)a-\delta_i}{2(d-1)}  \eqno(7)
$$
with $a:=d-1+c$ and the function $f$ assumes the value
$$
\begin{aligned}
f_{\rm max}& =  f(N_{\rm crit})=-\frac{d-1}{2} N_{\rm crit}^2\\
& =   m \left( \frac{c^2}{8(d-1)}+\frac{c}{4}+ \frac{d-1}{8}\right) + 
\frac{1}{8(d-1)}\sum_{i=1}^m (\delta_{i-1}-\delta_i)^2  \, . \\
\end{aligned}
$$
The function $f$ is bounded from above since the order of $\mathfrak{m}_{N}$ 
along $\Sigma$ is non-negative, so the critical point must give a maximum.
Note that the function $N^2$ is negative semi-definite since $N$ corresponds to a divisor
supported on (singular) fibres.
But the numbers given by (7) are not necessarily integral and the critical point
$N_{\rm crit}$
lives in (the subgroup generated by the components of our fibre in) 
${\rm Pic}(\widetilde{\PP})\otimes {\QQ}$. 
Therefore we replace the rational number $a_i$ given in (7) by
a near integer $\alpha_i$ in the following way: 
first we choose $\alpha_{m-1}\in {\ZZ}$
such that 
$$
\alpha_{m-1}=a_{m-1}+e_{m-1} \qquad \hbox{\rm with} \quad |e_{m-1}|\leq 1/2
$$
and then we choose successively the integer 
$\alpha_i$ for $i=m-2,\ldots,1$ such that
$$
\alpha_{i}=a_i+e_i \qquad \hbox{\rm with } 
\quad |e_{i}-e_{i+1}|\leq 1/2\, . \eqno(8)
$$
The function $f$, when viewed as a function of the coordinates of $N$, 
assumes at $\alpha=(\alpha_0,\ldots,\alpha_{m-1})$ the value 
$$
f(\alpha)= f_{\max}
-\frac{d-1}{2} \sum_{i=1}^m (e_{i-1}-e_i)^2 \, .
$$
\begin{lemma}\label{auxlemma1}
The value of $f$ at integral points takes its maximum at $\alpha$.
\end{lemma}
\begin{proof}
Let $h=(h_0,\ldots,h_{m-1}) \in {\ZZ}^m$. Then we have
$$
f(\alpha+h)=f(\alpha)-\frac{d-1}{2} \sum_{i=1}^m 2(e_{i-1}-e_i)(h_{i-1}-h_i)
+(h_{i-1}-h_i)^2 \, ,
$$
where we set $h_m=0$. Writing $u_i=h_{i-1}-h_i$ and $\epsilon_i=e_{i-1}-e_i$,
we see that $f(\alpha+h)-f(\alpha)
=-((d-1)/2) \sum_{i=1}^m u_i^2+2\, \epsilon_iu_i$
and since $|\epsilon_i|\leq 1/2$ by our choice of the $e_i$, this quadratic form
in the $u_i$ takes only non-positive values in integral points $u_i$.
\end{proof}
\begin{lemma}\label{auxlemma2}
For $N=\sum_{i=0}^{m-1} \alpha_i R_i$ we have 
that $F_{A+(d-1)N}=0$.
\end{lemma}
\begin{proof}
The divisor $A+(d-1)N$ has the form $\sum_{i=0}^{m-1} (d-1)[(m-i)/2+e_i] R_i$.
But $e_i \geq -(m-i)/2$, since we started with $e_{m-1}\geq -1/2$ and in
 each step we have $e_i\geq e_{i+1}-1/2$ by (8). So all coefficients are
$\geq 0$ and the last one is zero.
\end{proof}
Furthermore, we calculate the values of $f$ at $\theta_p$ 
and the trivial line bundle and find that
$$
f(\theta_p)=\frac{mc}{2} \qquad \hbox{\rm and}\qquad f(0)=0\, .
\eqno(9)
$$
Note that for $a_i=\alpha_i$ we have
$$
(F_A,F_{A+(d-1)N}) =
\begin{cases}
(-\frac{mc}{2} F, 0) & c> 0 , \\
(0,0) & c\leq 0, \\
\end{cases} \eqno(10)
$$
with $F$ denoting the full fibre.
Lemma \ref{auxlemma1} together with (9) and (10) shows that if we write 
$F_A-F_{A+(d-1)N}=(f_A-f_{A+(d-1)N})F$ with $F$ standing for a fibre we have
$$
0 \leq f_A-f_{A+(d-1)N} + \begin{cases}
f(\theta_p) & c>0 \\
f(0,\ldots,0) & c\leq 0\, . \\
\end{cases} \eqno(11)
$$
Collecting these facts we arrive at the following conclusion.

\begin{theorem}\label{correction1}
Let $\Sigma$ be an irreducible component of the boundary $S_{j,\mu}$ of
$\overline{\mathcal{H}}_{d,g}$. Then
the coefficient $\sigma_{j,\mu}$ of $\Sigma$ 
of the locus $\cap_N \mathfrak{m}_N$ is  equal to 
$$
\begin{aligned}
 m(\mu)\left(\frac{1}{12}(d-\sum_{\nu=1}^{n(\mu)} \frac{1}{m_{\nu}})+
\frac{j(b-j)(d-2)}{8(b-1)(d-1)}\right)
-\frac{1}{8(d-1)}\sum_{i=1}^{m(\mu)} (\delta_{i-1}-\delta_i)^2 &\\
-\frac{d-1}{2} \left( \frac{m(\mu)}{4} - \sum_{i=1}^{m(\mu)} (e_{i-1}-e_i)^2 
\right) \,& . \\
\end{aligned}
$$
\end{theorem}
\begin{proof}
Let $\Sigma$ be an irreducible component of the boundary $S_{j,\mu}$ of
$\overline{\mathcal{H}}_{d,g}$. 
Then for the choice of $N=\sum_{i=0}^{m-1} \alpha_i R_i$ given by (8) we have
$\mathfrak{m}_{\rm st} -\mathfrak{m}_{N}= \sigma \Sigma$
with $\sigma$ given by
$$
f_{\max}-\frac{d-1}{2}\sum_{i=1}^m (e_{i-1}-e_i)^2 + 
\begin{cases} mc/2\, & c>0 \\ 0  & c \leq 0 \, .  \\
\end{cases}  \eqno(12)
$$
Since this is non-negative by (9) and (11) the divisor class
$\mathfrak{m}_{N}$ represents the class of an effective divisor
containing the Maroni locus
that is smaller by $\sigma \Sigma $ than the effective divisor representing 
$\mathfrak{m}_{\rm st}$ given in Theorem \ref{classMaroni}. 
\end{proof}

\begin{remark}
Note that the correction term in the formula satisfies
$$
 \frac{m(\mu)}{4}-  \sum_{i=1}^m (e_{i-1}-e_{i})^2 \geq 0\, .
$$
\end{remark}

The contribution $\sigma$ in (12) depends only on the
pair $(j,\mu)$ and for $j$ only on the residue class mod $2(d-1)$ and it is a 
small correction. We list for $3\leq d \leq 5$ the pairs $(j\, (\bmod\, 2(d-1)), \mu)$ for which we find a positive correction $\sigma(j,\mu)$.

\begin{footnotesize}
\smallskip
\vbox{
\bigskip\centerline{\def\quad{\hskip 0.6em\relax}
\def\quod{\hskip 0.5em\relax }
\vbox{\offinterlineskip
\hrule
\halign{&\vrule#&\strut\quod\hfil#\quad\cr
height2pt&\omit&&\omit&&\omit&&\omit&&\omit&&\omit&&
\omit&&\omit&\cr
&$\mu$&&$[3]$&&$[4]$&&$[4]$&&$[3,1]$&&$[2,2]$&&$[5]$&&$[5]$&\cr
\noalign{\hrule}
&$j$&&$0$&&$1$&&$5$&&$0$&&$0$&&$0$&&$2$&\cr
\noalign{\hrule}
&$\sigma$&&$1$&&$1$&&$1$&&$1$&&$1$&&$2$&&$1$&\cr
\noalign{\hrule}
\omit&&\omit&&\omit&&\omit&&\omit&&\omit&\cr
\noalign{\hrule}
height2pt&\omit&&\omit&&\omit&&\omit&&\omit&&\omit&&
\omit&&\omit&\cr
&$\mu$&&$[5]$&&$[4,1]$&&$[4,1]$&&$[3,2]$&&$[3,2]$&&$[3,1,1]$&&$[2,2,1]$&\cr
\noalign{\hrule}
&$j$&&$6$&&$1$&&$7$&&$1$&&$7$&&$0$&&$0$&\cr
\noalign{\hrule}
&$\sigma$&&$1$&&$1$&&$1$&&$1$&&$1$&&$1$&&$1$&\cr
} \hrule}
}}
\end{footnotesize}

\centerline{Table 1.}

\end{section}
\begin{section}{Improving the Maroni Class by Variation of 
${\mathcal L}$ and Twisting}\label{varyingL}

In this section we investigate the behaviour  of $\mathfrak{m}_{{\mathcal L},N}$
when we also vary ${\mathcal L}$. 
We write $\mathcal{L}={\mathcal O}_{\tilde{Y}}(Z)$ with
$Z$ an effective divisor supported on the boundary of $\tilde{Y}$.

In order to calculate the difference 
$\mathfrak{m}_{\rm st}-\mathfrak{m}_{\mathcal{L},N}$
we need to calculate some Chern classes.  
Recall that for the effective line bundle $\mathcal{L}$
the reflexive sheaf $V_{\mathcal{L}}$ is defined as the dual
of the cokernel $K_{\mathcal{L}}$ of the natural map
$\mathcal{O}_{\widetilde{\PP}} \to \tilde{\pi}_*\mathcal{L}$. 
We fix the notation by
$$
V:= V_{{\mathcal O}_{\widetilde{Y}}}, \quad V^{\prime}=V\otimes M,
\quad \hbox{\rm and}  
\quad V_{1}^{\prime}=V_{\mathcal{L}}\otimes M \otimes N,
$$
where $N$ corresponds to a divisor class supported on the boundary 
of $\widetilde{\PP}$ and $M$ is defined in \ref{standardM}.
Note that in case of $\mathcal{L}=\mathcal{O}_{\widetilde{Y}}$
the results of section \ref{goodmodel} show that the
direct image $\tilde{\pi}_*\mathcal{O}_{\widetilde{Y}}$ is a vector bundle.

An application of Grothendieck-Riemann-Roch (to $\tilde{\pi}$, $\mathcal{L}$
and $\mathcal{O}_{\widetilde{O}}$) allows us to compare the
Chern classes of $\tilde{\pi}_*\mathcal{O}_{\widetilde{\PP}}$ 
and $\tilde{\pi}_*\mathcal{L}$. We leave the proof to the reader.

\begin{proposition}\label{chernclasses} For $\mathcal{L}=\mathcal{O}_{\tilde{Y}}(Z)$
we have with $U$ the ramification locus of $\tilde{\pi}$:
$$
\begin{aligned}
&c_1(\tilde{\pi}_*{\mathcal L})=\tilde{\pi}_*(Z)-\frac{1}{2}\tilde{\pi}_*(U), \\
&c_2(\tilde{\pi}_*{\mathcal L})-c_2(\pi_* {\mathcal O}_{\tilde{Y}})=
\frac{1}{2} ((\tilde{\pi}_*Z)^2-\tilde{\pi}_* Z^2) 
-\frac{1}{2}(\tilde{\pi}_*U\cdot \tilde{\pi}_*Z-\tilde{\pi}_*(U\cdot Z)). 
\\
\end{aligned}
$$
\end{proposition}

Now  we make the following assumption.

\begin{assumption}\label{assumptions} We assume:
\begin{enumerate}
\item{} The effective divisor $Z$ on $\widetilde{Y}$ is
supported on the boundary divisor and $\mathcal{L}=\mathcal{O}(Z)$ 
is the pull back of an effective line bundle on $Y$.
\item{} The image of the section $1$ under 
$\iota_{\mathcal{L}}:
\mathcal{O}_{\widetilde{\PP}} \to \tilde{\pi}_*\mathcal{L}$ 
vanishes nowhere on $\widetilde{\PP}$ 
\end{enumerate}
\end{assumption}

Note that by \ref{assumptions} (1) if 
$\mathcal{L}=\nu^* \mathcal{L}^{\prime}$, the sheaf
$\tilde{\pi}_* \mathcal{L}=\pi_* \mathcal{L}^{\prime}$ 
is locally free since $\pi$ is a finite flat morphism.
The second assumption implies in particular 
that $Z$ does not contain a divisor of the form $\tilde{\pi}^* D$
for an effective divisor $D$ on $\widetilde{\PP}$.

\begin{proposition}
Under the assumptions \ref{assumptions} we have 
$c_i((\pi_*\mathcal{L})^{\vee})=c_i((K_{\mathcal{L}})^{\vee})$ 
for $i=1,2$.
\end{proposition}
\begin{proof} 
This follows because $\tilde{\pi}_*{\mathcal L}$ is locally free
and the image of $\iota$ defines a non-zero section, hence a
locally free direct summand of rank $1$.
\end{proof}

By our assumption $\tilde{\pi}_*\mathcal{L}$ is locally free, hence
Proposition \ref{chernclasses} tells us the Chern classes of 
$\tilde{\pi}_* \mathcal{L}$ and hence of its dual.
We thus find
$$
\begin{aligned}
c_2(V_{\mathcal{L}}\otimes N)-c_2(\tilde{\pi}_*{\mathcal O}_{\tilde{Y}})=
\frac{1}{2} \left( (\tilde{\pi}_*Z)^2 -\tilde{\pi}_*(Z^2) \right)
-\frac{1}{2} \left(\tilde{\pi}_*U \cdot \tilde{\pi}_* U -\tilde{\pi}_*(U\cdot Z)  \right)  +& \\ 
(d-2)(-\tilde{\pi}_*Z+\frac{1}{2} \tilde{\pi}_* U) \cdot N + \binom{d-1}{2} N^2 &\\
\end{aligned}
$$
and we observe that $\tilde{\pi}_*(U)=W$ and 
$c_1(V_{1}^{\prime})-c_1(V^{\prime})= (d-1) N - \tilde{\pi}_*Z$.
We set 
$$
A_{1}=A_{\rm st} + (d-1)N-\tilde{\pi}_*Z
$$
and then have $c_1(V_{1}^{\prime})=A_{1}+p^*D$ 
for some divisor class $D$ on the base as in (2).  

If we calculate the expression $(1/2)A\cdot \theta_p -(1/2(d-1))A^2$ for
both $V^{\prime}$ and $V^{\prime}_{1}$ we find as difference
$$
\frac{1}{2(d-1)} G \cdot (G+2A-(d-1)\theta_p)\, ,
$$
where we write $A$ for $A_{\rm st}$ and use
$$
G:= (d-1)N-\tilde{\pi}_*Z\, .
$$
Now Theorem  \ref{firstformula} gives the following result about the difference
$\mathfrak{m}_{\rm st}-\mathfrak{m}_{\mathcal{L},N}$.

\begin{theorem}\label{differenceLN}
Under Assumption \ref{assumptions} 
with $\mathcal{L}=\mathcal{O}_{\widetilde{Y}}(Z)$ 
the difference $\mathfrak{m}_{\rm st}-\mathfrak{m}_{\mathcal{L},N}$
is the pull back under a section of $p$ of
$$
\begin{aligned}
F_{A}-F_{A_1}+
\frac{1}{2}\tilde{\pi}_*(Z^2-Z\cdot U)-\frac{1}{2(d-1)} \tilde{\pi}_*Z \cdot
(\tilde{\pi}_*Z -\tilde{\pi}_*U)+ \qquad &\\
\frac{1}{2(d-1)} G \cdot (G+2A-(d-1)\theta_p) & \, , \\
\end{aligned}
$$
with $\tilde{\pi}_*U=W$ and $A=A_{\rm st}$.
\end{theorem}
Note that this does not change if we add a full fibre of $p$ 
to $N$.

If $\iota_{\mathcal{L}}(1)$ does not satisfy the condition of nowhere
vanishing then $K_{\mathcal{L}}$ may not be torsion-free. For example,
if $\mathcal{L}'$ is such that $\tilde{\pi}_*\mathcal{L}'$ is reflexive
and if $\mathcal{L}=\mathcal{L}' \otimes \mathcal{O}(\tilde{\pi}^*D)$
for an effective divisor $D$ on $\widetilde{\PP}$ then by the projection
formula we have $\tilde{\pi}_* \mathcal{L}= \tilde{\pi}_*\mathcal{L}' \otimes 
\mathcal{O}(D)$ and we have a commutative diagram
\begin{displaymath}
\begin{xy}
\xymatrix{
& 0 \ar[d] & 0 \ar[d] \\
& \mathcal{O} \ar[r]^{=} \ar[d] & \mathcal{O}\ar[d] \\
0 \ar[r] & \mathcal{O}(D) \ar[r] \ar[d] & \tilde{\pi}_*\mathcal{L} \ar[d] \ar[r] & K_{\mathcal{L}'}\otimes O(D) \ar[d]^{\cong}\ar[r] & 0 \\
0 \ar[r] & \mathcal{O}_D(D) \ar[r]\ar[d] & K_{\mathcal{L}} \ar[r] \ar[d] & Q \ar[r] & 0 \\
& 0 & 0 \\
}
\end{xy}
\end{displaymath}
showing that $O_D(D)$ is the torsion subsheaf of $K_{\mathcal{L}}$.
Then the formula of Theorem \ref{differenceLN} needs a correction due to 
the transition to the dual.  In fact, if  
$Z_1=Z+\tilde{\pi}^*D$, with $Z$ satisfying the assumptions \ref{assumptions}
and $D$ effective then $\mathfrak{m}_{\mathcal{O}(Z_1),N}=\mathfrak{m}_{\mathcal{O}(Z),N(-D)}$.

We now wish to maximize this difference given in Theorem \ref{differenceLN}. 
Discarding the term 
$F_A-F_{A_1}$ we have a quadratic function $f$ of $Z$ 
and $G$ which we can write as $f=f_1(Z)+f_2(G)$ with
$$
f_1=\frac{1}{2}\tilde{\pi}_*(Z^2-Z\cdot U)-\frac{1}{2(d-1)} \tilde{\pi}_*Z \cdot
(\tilde{\pi}_*Z -\tilde{\pi}_*U)
$$
and
$$
f_2= \frac{1}{2(d-1)} G \cdot (G+2A-(d-1)\theta_p) 
$$
and we can maximize the two parts independently. 
The function $f_2$ is a quadratic
function of $G$, a divisor class on $\widetilde{\PP}$, 
while the function $f_1$ is a quadratic function of $Z$, a divisor class
in the semi-group of the Picard group of $\widetilde{Y}$
generated by the effective divisors supported on the singular fibres. 
We can then consider the situation
locally around a fibre of $\widetilde{Y}$ for $f_1$ and locally around 
a (corresponding) fibre of $\widetilde{\PP}$ for $f_2$. 
The function $f_2$ (resp.\ $f_1$) is then a quadratic
function on the subgroup (resp.\ semi-group) 
of the Picard group generated by the components of that fibre.
For $f_2$ we observe that $f_2$ remains unchanged if we add to the divisor class
of $N$ a multiple of the fibre. Therefore we can consider the function $f_2$
on the subgroup generated by the components of the fibre minus one component.
We leave out the component intersected by a section and since $\theta_p$ is
represented by $-2S$ plus components from the fibre we can represent $\theta_p$
by a divisor with support on the fibre.
Then we can find, similarly to what we did in Section \ref{twistingV}, 
the maximum by looking for a critical point of the quadratic function by 
differentiating and get
$$
G+A -\frac{d-1}{2}\theta_p=0.
$$
Then $G=-A+(d-1)\theta_p/2$ defines using the intersection 
form a functional on the subgroup
of the Picard group generated by the components of the fibre 
and this functional can be
represented by a rational linear combination of the fibre components, 
uniquely up to addition of a multiple
of the fibre. 
Since it is a rational combination we can approximate it by an integral 
linear combination that is optimal. 
Since we know that the order of $\mathfrak{m}_{\mathcal{L},N}$
along any boundary component $\Sigma$ is non-negative 
we see that the quadratic function $f_2$ is
bounded from above and the critical point gives the maximum of $f_2$.

Next we look at the function $f_1$ as quadratic function 
(of $Z$) on the semi-group $\Gamma$
generated by the components of the fibre, 
this time of $\widetilde{Y}$. 
By our assumption \ref{assumptions} (2) 
we have excluded the case that $Z$ is the full preimage of an effective
divisor on $\widetilde{\PP}$. In fact,  using such $Z$ boils down
to tensoring with a line bundle $N$ on $\widetilde{\PP}$ as we saw above. 
But we now also have the constraint that $Z$ is effective. 
Therefore the method of using the critical point
may not yield the maximum.  The obvious critical point of $f_1$
given by $Z=U/2$ may not represent an effective 
divisor with support on the fibre
and a critical point may not represent the maximum. 
In any case, either the critical point is 
effective and gives a maximum, or the maximum is 
to be found on the boundary of 
the closure of the rational polyhedral cone generated by $\Gamma$ 
in the ${\QQ}$-vectorspace generated by $\Gamma$.
Again, the fact that the order of $\mathfrak{m}_{\mathcal{L},N}$ 
is non-negative guarantees that $f_1$ has a maximum on the 
subset of the semi-group with rational coefficients generated by $\Gamma$
that is allowed by our assumptions \ref{assumptions}.
The situation  depends very much on the parity 
and divisibility properties of the (intersection) 
numbers involved. 
We give one example.
\begin{example}
We consider the trigonal case $d=3$, $g$ even, and 
an irreducible component $\Sigma$ of $S_{j,[1,1,1]}$ whose generic 
point corresponds to the following 
picture

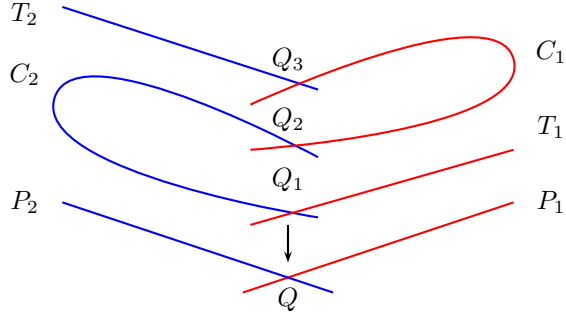
\begin{center}
\begin{pspicture}(-3,-2)(3,3)
\psline[linecolor=red](-0.6,-1.2)(3,0) 
\psline[linecolor=blue](-3,0)(0.6,-1.2) 
\pscurve[linecolor=blue](0.4,-0.2)(-3.1,1.4)(0.4,0.6)
\psline[linecolor=blue](-3.0,2.6)(0.4,1.5) 
\psline[linecolor=red](-0.5,-0.3)(3,0.7) 
\pscurve[linecolor=red](-0.5,0.7)(3,1.9)(-0.5,1.3) 
\psline{->}(0,-0.3)(0,-0.8)
\rput(-3.5,1.7){$C_2$}
\rput(-3.5,2.5){$T_2$}
\rput(3.5,2){$C_1$}
\rput(3.5,1){$T_1$}
\rput(-3.5,0){${P}_2$}
\rput(3.5,0){${P}_1$}
\rput(0,1.9){$Q_3$}
\rput(0,1.1){$Q_2$}
\rput(0,0.3){$Q_1$}
\rput(0,-1.3){$Q$}
\end{pspicture}
\end{center}

\noindent
We have components $C_i$ (resp.\ $T_i$) mapping to $P_i$ 
with degree $2$ (resp.\ degree $1$) for $i=1,2$. The
intersection behavior is $C_i^2=-2$,
$T_i^2=-1$ and $T_1T_2=0$, $C_1C_2=1$, $C_1T_2=1=C_2T_1$, and $C_1T_1=0=C_2T_2$.
Given our conditions \ref{assumptions} and the symmetry we consider 
effective ${\QQ}$-linear 
combinations $a_1C_1+b_2T_2$. In the ${\QQ}$-vectorspace
$\langle C_1,T_2\rangle$ 
of such linear combinations we can represent
$U$ by the non-effective cycle $-lC_1-lT_2$ since $UC_1=l$ and $UT_2=0$.
The quadratic form $(1/2)\tilde{\pi}_*(Z^2-ZU)-(1/(2(d-1))(\tilde{\pi}_*(Z)
(\tilde{\pi}_*Z-\tilde{\pi}_*U)$ takes the form 
$b_2(b-l-b_2)/4$
with critical value $(b-l)^2/16$ for $b_2=(b-l)/2$. 
If we look on  the ${\QQ}$-space
$\langle C_1\rangle$ we get the zero form while on the space 
$\langle T_1\rangle$ the quadratic form is $b_1(l-b_1)/4$ with critical point
$b_1=l/2$ giving the value $l^2/16$ and either this or $(b-l)^2/16$ 
represents indeed the effective maximum of the function $f_1$ in this case.
\end{example}

In the next section we will work out one specific case where we can apply
Theorem \ref{differenceLN} and of which the above example is a special case.
\end{section}
\begin{section}{A Special Case}\label{specialcase}
We shall apply the preceding in the case where one of the ramification
indices $m_{\nu}$ equals $1$. We shall assume $d\geq 3$.
The results will be used in the Section \ref{comparison}.

So we look again locally on our 
$1$-dimensional base $B$  near a point $s$ of an irreducible divisor
$\Sigma$ of the boundary $S_{j,\mu}$ 
of our Hurwitz space and we want to calculate
$$
{\rm ord}_{\Sigma}\mathfrak{m}_{\rm st}
-{\rm ord}_{\Sigma}\mathfrak{m}_{\mathcal{L},N}
$$
for a suitable choice of divisors $Z$ and $N$.

We suppose that $\mu$ is such that one of the indices $\mu_{\nu}$ equals $1$.
If the ramification index $m_{\nu}=1$ for a point $Q_{\nu}$ on the fibre 
of $\tilde{C}$ over $s$ then locally near the preimage under the map
$\widetilde{Y}\to \mathcal{C}$ of the point $Q_{\nu}$ 
(see the beginning of Section \ref{goodmodel})
our space $Y$ is smooth according
to Lemma \ref{normalizationlemma}. 
The fibre of $p$ over $s$ is a chain of
rational curves $R_0=P_1, R_1=E_1,\ldots, R_{m-1}=E_{m-1}, R_m=P_2$
 and above it in $Y$ and $\widetilde{Y}$ we have a corresponding chain $T_0,\ldots,T_m$
such that $T_i\to R_i$ is unramified of degree $1$ for $i=1,\ldots,m-1$ because
of $\mu_{\nu}=1$; 
we further require that $T_0$ is a rational tail and  $T_0\to R_0$ is unramified of degree $1$.
The resolution map $\nu: 
\widetilde{Y} \to Y$ does not affect the local situation and hence the 
proper transforms of the $T_i$ do not intersect the exceptional locus of $\nu$.
We now choose 
$$
N=\sum_{i=0}^{m-1} a_iR_i \quad \hbox{\rm and}\quad  
Z=\sum_{i=0}^{m-1} b_i T_i \quad \hbox{\rm so that} \quad\tilde{\pi}_*Z=
\sum_{i=0}^{m-1} b_i R_i\, ,
$$ 
where the coefficients $b_i$ are non-negative.

We have $T_0^2=R_0^2=-1$ while $T_j^2=R_j^2=-2$ for $1\leq j\leq m-1$;
Moreover, since the maps $T_i \to R_i$ ($i=0,\ldots,m-1$) are unramified 
of degree $1$ we have $Z \cdot U=0$ and by the projection formula 
we have $\tilde{\pi}_*(Z^2)=(\tilde{\pi}_*Z)^2$.
Thus we see that the contribution of Theorem \ref{differenceLN}
is $f_A-f_{A_1}+f(Z,N)$ with $f(Z,N)$  the quadratic expression
$$
f(Z,N):=\frac{1}{2(d-1)}\left( (d-2)X^2+WX+G(G-(d-1)\theta_p+2A)\right) \, ,
$$
where we used $X=\tilde{\pi}_*Z$ and $W=\tilde{\pi}_*(U)$. Discarding again 
$F_A-F_{A_1}$, we look for the critical point of the function $f(Z,N)$
viewed as a function of the variables $X$ and $G$. 
The vanishing of the partial derivatives to $X$ and $G$ is given by
$$
2(d-2) \, X +  W=0 \qquad \hbox{\rm and} \qquad 
2\, G -(d-1)\theta_p+2\, A=0\, .
$$
So the critical point is given by
$$
(X,G)=(-W/2(d-2),(d-1) \theta_p/2 -A)\, .
$$
This means that this pair $(X,G)$ defines on the subgroup of 
${\rm Pic}(\widetilde{\PP})$ generated by the components $R_i$ ($i=0,\ldots,m-1$)
of our fibre the same intersection behavior as the pair of critical point 
$(\tilde{\pi}_*Z, (d-1)N-\tilde{\pi}_*Z)$ we are looking for.
The value at the critical point is given by
$$
-\frac{d-1}{8} \theta_p^2 +
\frac{1}{2} A \theta_p -\frac{1}{2(d-1)}A^2-\frac{1}{8(d-1)(d-2)}W^2\, .
$$
Observe that near our fibre $W=\Xi+\sum_{i=0}^{m-1}\delta_i R_i$
with $\Xi$ the sum of the sections, 
and therefore in the subgroup of  $\widetilde{\PP}$ generated by
$R_0, \ldots ,R_{m-1}$ we have 
$W = -\sum_{i=1}^{m-1}[(m-i)l-\delta_i]R_i$. Moreover 
 $A=-(1/2)\sum_{i=0}^{m-1} ((m-i)c-\delta_i)R_i$
and $\theta_p=\sum_{i=0}^{m-1} (m-i)R_i-2S$ 
with $S$ the section that intersects $R_m$ only. 
The coordinates of the critical
point in terms of the coordinates $g_i$ of $G$ and $x_i$ of $X$ 
(for $i=0,\ldots,m-1$) are 
$$
g_i=\frac{(m-i)a-\delta_i}{2}
\qquad \hbox{\rm and} \qquad x_i= \frac{(m-i)l-\delta_i}{2(d-2)}\, ,\qquad
\hbox{\rm for $i=0,\ldots,m-1$}\, ,
$$
with $a=c+d-1$. Observe that the $x_i$ are non-negative.
The value taken by $f$ at the critical point is
$$
f_{\max}=\frac{m(\mu)}{8(d-1)(d-2)}\left((b-j)^2+(d-2)a^2\right)
+\frac{1}{8(d-2)}\sum_{i=1}^{m(\mu)} (\delta_{i-1}-\delta_i)^2\, . \eqno(14)
$$
We let $N=\sum_{i=1}^{m-1} a_i R_i$. We then have $(d-1)a_i=g_i+x_i$.
Since the $a_i$ and $x_i$ are not necessarily integers we define 
integers $\alpha_i$ and $\xi_i$ by
$$
\alpha_i=a_i +e_i, \qquad
\xi_i=x_i +e^{\prime}_i\quad \hbox{\rm for } i=0,\ldots,m-1\, ,
\eqno(15)
$$
where the  $e_i$ and $e_i^{\prime}$ are chosen as follows.
We choose $e_{m-1}\in [-1/2,1/2)$ and successively we choose $e_i$
for $i=m-2,\ldots,1$ such that $|e_i-e_{i+1}|\leq 1/2$.
By noting that  $a-l=(d-1)(2(q-k-1)+1)$,
with $k=g/(d-1) $ and $q$ given in (4), we observe that
$$
a_i-x_i= \frac{(m-i)(a-l)}{2(d-1)} = \frac{(m-i)(2(q-k-1)+1)}{2} \, ,
$$
which is an integer or a half-integer according to the parity of $m-i$.
We then choose the $e^{\prime}_i$ such that 
$e^{\prime}_i= e_i$  for $m-i$ even
and $e^{\prime}_i= e_i+1/2$  for $m-i$ odd. 
The value of $f$ at the point
$(\alpha_0, \ldots, \alpha_{m-1},\xi_0, \ldots , \xi_{m-1})$ is 
$f_{\rm max}+f_{\rm err}$ with $f_{\rm err}$ given by
$$
-\frac{d-1}{2}\sum_{i=1}^m(e_{i-1}-e_i)^2+
\sum_{i=1}^m(e_{i-1}-e_i)(e^{\prime}_{i-1}-e^{\prime}_i)
-\frac{1}{2}\sum_{i=1}^m(e^{\prime}_{i-1}-e^{\prime}_i)^2
$$
which simplifies to
$$
-\frac{d-2}{2}\sum_{i=1}^m(e_{i-1}-e_i)^2-
\frac{1}{2}\sum_{i=1}^m((e_{i-1}-e_i)-(e^{\prime}_{i-1}-e^{\prime}_i))^2 \, ,
\eqno(16)
$$
so that we get
$$
f_{\rm err}= -\frac{d-2}{2}\sum_{i=1}^m(e_{i-1}-e_i)^2-\frac{m}{8} \, .
\eqno(17)
$$
\begin{lemma}
The maximum of $f$ in integral points is taken at $(\alpha,\xi)$.
\end{lemma}
\begin{proof}
For any choice of rational numbers $e_i,e_i^{\prime}$ and with 
$\alpha_i$ and $\xi_i$ as
in (15) we have $f(\alpha,\xi)=f_{\max}+f_{\rm err}$ with $f_{\rm err}$
as in (16). We have to maximize $f_{\rm err}$. Using $t_i=e_{i-1}-e_i$ and
$s_i=e_{i-1}^{\prime}-e_i^{\prime}$ we can rewrite it as
$$
-\frac{d-2}{2}\sum_{i=1}^mt_i^2-\frac{1}{2}\sum_{i=1}^m(t_i-s_i)^2\, .
$$
The term $\sum_{i=1}^m(t_i-s_i)^2$ is minimized 
independently of the choice of the $e_i$  by $m/8$
by choosing $e_i = e^{\prime}_i$ or $e_i = e^{\prime}_i- 1/2$ 
depending on the parity of $m-i$. Then by the
argument of Lemma \ref{auxlemma1} $f_{\rm err}$ is maximized 
by our choice of the $e_i$.
\end{proof}
\begin{lemma}\label{fA1}
For $(\alpha,\xi)$ as defined above we have $f_{A+G}=0$.
\end{lemma}
\begin{proof} We have
$$
A+G=\sum_{i=0}^{m-1} (\frac{(m-i)(d-1)}{2}+ (d-2)e_i-\epsilon_i)R_i 
$$
with $\epsilon _i=0$ if $m-i$ is even and $\epsilon_i=1/2$ otherwise. Since
$(d-2)e_i-\epsilon_i \geq  -(d-1)/2$,  
the coefficients are non-negative while the coefficient of $R_m$ is zero.
\end{proof}
\begin{theorem}\label{correction2}
Let $\Sigma$ be an irreducible  component of the boundary $S_{j,\mu}$
of $\widetilde{\mathcal{H}}_{d,g}$ corresponding to
the case $(j,\mu)$ with at least one $m_{\nu}$ equal to $1$.
We also assume that $T_0\to R_0$ is unramified of degree $1$.
Then the coefficient $\sigma_{j,\mu}$ of $\Sigma$ of the locus
$\cap_{\mathcal{L},N}\mathfrak{m}_{\mathcal{L},N}$
is equal to
$$
\begin{aligned}
m(\mu)\left(\frac{1}{12}(d-\sum_{\nu=1}^{n(\mu)} \frac{1}{m_{\nu}})+
\frac{j(b-j)(d-2)}{8(b-1)(d-1)}\right)-\frac{1}{8(d-2)}
\sum_{i=1}^{m(\mu)}(\delta_{i-1}-\delta_i)^2 &\\
-\frac{m(\mu)(b-j)^2}{8(d-1)(d-2)}
- \frac{d-2}{2} \left(\frac{m(\mu)}{4}-
\sum_{i=1}^{m(\mu)}(e_{i-1}-e_i)^2 \right)\, .  &\\
\end{aligned}
$$
\end{theorem}
\begin{proof}
By choosing $N$ and $\mathcal{L}$   corresponding to the near critical
points $(\alpha, \xi)$ given in (15) we see that we can 
diminish the coefficient
of $\Sigma$ in $\mathfrak{m}_{\rm st}$
by $f_A-f_{A_1}+f_{\max}+f_{\rm err}$. For the term $f_{A}-f_{A_1}$
the identity (10) remains valid by Lemma \ref{fA1}.
Also the positivity argument given there
can be used: we have $f((d-1)\theta_p, 0)=mc/2$ and $f(0,0)=0$.
Collecting all the terms gives the result.
\end{proof}
\bigskip
The same method can also be applied in more complicated situations. For
example, we may assume that ${\PP}$ has a singularity at $Q=P_1\cap P_2$
of the type $s^m=uv$ and that over it in the fibre of $Y$ at $Q_{\nu}$
we have a singularity of type $s^{m/m_{\nu}}=x_{\nu}y_{\nu}$. Then on
$\widetilde{\PP}$ we find a chain $R_0=P_1,R_1,\ldots,R_{m-1},R_m=P_2$
and on $Y$ we find a chain with $T_{\nu,i}$ with $i=1,\ldots,m-1$
with $T_{\nu,i}$ mapping to $R_i$ with degree $d_{\nu,i}={\rm gcd}(m_{\nu},i)$
and ramification $m_{\nu}/d_{\nu,i}$. Let $C_{\nu}$ be the component lying
over $R_0$ containing the point $Q_{\nu}$ and assume for simplicity 
that it contains no other singular points. 
Then $\pi^*(R_i)=\sum_{\nu} T_{\nu,i}^{\prime}$ 
with $T_{\nu,i}^{\prime}=(m_{\nu}/d_{\nu,i})T_{\nu,i}$ a Cartier divisor.
We can then apply intersection theory on the normal surface $Y$ (for example
we have the projection formula for $\pi: Y \to \tilde{\PP}$ and $\nu: 
\widetilde{Y}\to Y$; see \cite{EdeJS}, p.\ 867) and one calculates
$$
{T_{\nu,i}^{\prime}}^2=-2m_{\nu}, \quad
C_{\nu}T_{\nu,i}^{\prime}=T_{\nu,i}^{\prime} T_{\nu,i+1}^{\prime}=m_{\nu}\, .
$$
If we now put 
$Z_1=a_{\nu,0}C_{\nu}+\sum_{i=1}^{m-1} a_{\nu,i} T_{\nu,i}^{\prime}$
and set $Z=\nu^*Z_1$ then these are effective Cartier divisors and
we can calculate
$$
Z^2= -m_{\nu} \sum_{i=1}^m (a_{\nu,i-1}-a_{\nu,i})^2, \quad
\nu^*T_{\nu,i}^{\prime} U=-d_{\nu,i-1}+2d_{\nu,i}-d_{\nu,i+1}\, ;
$$
moreover we have
$$
\nu^* C_{\nu} \cdot U=3(m_{\nu}-1)+2g(C_{\nu})
$$
together with
$$
(\tilde{\pi}_*Z)^2=-\sum_{i=1}^m (d_{\nu,i-1}a_{\nu,i-1}-d_{\nu,i}a_{\nu,i})^2
$$
and
$$
\tilde{\pi}_*Z \cdot W= a_{\nu,0}m_{\nu}l-\sum_{i=1}^m 
(d_{\nu,i-1}a_{\nu,i-1}-d_{\nu,i}a_{\nu,i})(\delta_{i-1}-\delta_i)\, .
$$
With these formulas, similar in spirit to the above but more involved, we
can calculate the extended Maroni classes in such cases using Theorem
\ref{differenceLN}.
\end{section}
\begin{section}{A Question}
To a pair $(\mathcal{L},N)$ with $\mathcal{L}=\mathcal{O}_{\widetilde{Y}}(Z)$
for an effective divisor $Z$ supported on the boundary of $\widetilde{Y}$
and $N$ a line bundle  supported on the boundary of $\widetilde{\PP}$ we
associated the class of an effective divisor $\mathfrak{m}_{\mathcal{L},N}$
on $\overline{\mathcal{H}}_{d,g}$
containing the closure of the Maroni divisor. 
By varying $\mathcal{L}$ and $N$
we can define an effective divisor class
$$
\mathfrak{m}_{\min}= \cap_{\mathcal{L},N} \mathfrak{m}_{\mathcal{L},N}\, .
$$
This is the minimal class we can get with this method.
The formulas we
found for the extended Maroni classes $\mathfrak{m}_{\mathcal{L},N}$ are
given in terms of the boundary classes $S_{j,m}$. 
The rational Picard group of $\overline{\mathcal H}_{d,g}$ is not known in general,
but is conjectured to be generated by the boundary classes. 
In his thesis \cite{Patel,Patel-ArXiv} Patel proves that the boundary classes
on the unordered Hurwitz space are independent.
If the classes of the irreducible
boundary components 
$\Sigma$ are linearly independent 
we may describe this  minimal Maroni class
$\mathfrak{m}_{\min}$ by
$$
\mathfrak{m}_{\min}:= \sum_{\Sigma} \alpha_{\Sigma} \Sigma
$$
with $\alpha_{\Sigma}$ given by
$$
\alpha_{\Sigma}= \min \{ \alpha_{\Sigma}(\mathcal{L},N) : \mathcal{L},N\}\, ,
$$
where we write
$$
\mathfrak{m}_{\mathcal{L},N}=
\sum_{\Sigma} \alpha_{\Sigma}(\mathcal{L},N) \Sigma\, ,
$$
and with $\mathcal{L}=\mathcal{O}_{\widetilde{Y}}(Z) $  ranging over the 
line bundles associated to effective divisor classes with support on the boundary
and $N$ over the line bundles corresponding to divisor classes with support on
the boundary. The following natural question comes up.

\begin{question}
Does the class $\mathfrak{m}_{\min}$ coincide with the Zariski 
closure of the Maroni locus?
\end{question}
In the next section we show that in the trigonal case the answer is
positive. 
\end{section}
\begin{section}{Comparison with the Results of Patel and Deopurkar-Patel}\label{comparison}
In this section we compare our results with the results of Patel \cite{Patel} on a 
partial compactification and with the results of Deopurkar-Patel \cite{D-P}
for the trigonal case.

We begin by comparing with the results of Patel by 
 specializing our calculation to a partial
compactification of our Hurwitz space involving only the boundary divisors
$S_{j,\mu}$ of the Hurwitz space ${\overline{\mathcal H}_{d,g}}$ that
parametrize covers that
have $b-2$ ramification points on one ${\PP}^1$ and $2$ on the other;
in other words $j=2$. We then can compare our result with the results
of Patel \cite{Patel} on such a compactification.
The boundary components $S_{2,\mu}$ involved fall into three cases:
$$
\sum_{\mu} S_{2,\mu}= E_2+E_3+ \Delta,
$$
where $E_2$ collects the divisors corresponding to
$\mu=(2,2,1,\ldots,1)$, while $E_3$ collects those corresponding to
$\mu=(3,1,\ldots,1)$ and the  divisor $\Delta$ contains
all $S_{2,\mu}$ with $\mu=(1,\ldots,1)$.
Each generic point of the divisor $E_2$ corresponds to curves
with a $g^1_d$ with two simple ramification points in one fibre, while
each generic point of $E_3$ to those with a triple ramification point.
Each irreducible component of the divisors $E_2$ and $E_3$ maps dominantly
to a cycle in $\overline{\mathcal M}_g$ that intersects
the interior ${\mathcal M}_g$.
The divisor $\Delta$
includes the case of covers with a base point
(that is, the divisor of covers $C'\cup L\to {\PP}^1$
with $C'\to {\PP}^1$ of degree $d-1$ and $L\to {\PP}^1$ of degree $1$
and $L\cap C'$ one point, the `base point'), but also the one-nodal curves.

We calculate the contributions to these terms in Theorem \ref{classMaroni}
using the identity $b=2(k+1)(d-1)$:
$$
-\frac{(k+1)(d-2)}{2(b-1)}\, \Delta + \frac{2k+1}{2(b-1)} \, E_2
-\frac{(d-10)(k+1)+4}{6\, (b-1)} \, E_3 \, .
$$
Except for a factor $2$ this expression coincides with the expression that
Patel finds in his thesis for the class of an effective divisor
containing the Maroni locus on a partially
compactified Hurwitz space that takes into account 
only the case where $j=2$. The factor $2$ is due to the automorphism interchanging
the two points on one tail ${\PP}^1$. 

But as we shall show now, 
this is not the class of the closure of the Maroni locus.
To show this, we consider the case where $\mu=(1,\ldots,1)$.
Note that then $m(\mu)=1$ and our space $\widetilde{C}$ is smooth here.
There are many irreducible components of the divisor $S_{j,\mu}$
for $j=2$ and $\mu=(1,\ldots,1)$. 
We further assume that we have $2g-4+2d$ branch points on $P_1$ and $2$ on $P_2$.
One divisor contained in $S_{2,\mu}$ parametrizes covers where there are
two irreducible curves $C_1$ and $C_2$ of genus $g_1$ and $g_2$ 
covering $P_1$ with degrees $d_1$ and $d_2$ with $d_1+d_2=d$ 
and these are joined by a rational
tail mapping with degree $2$ to $P_2$. Over an appropriate
 $1$-dimensional base $B$ we have by the projection formula $C_i^2=-d_i$. 
We now choose $\mathcal{L}$ associated to the effective divisor $Z=y\, C_1$
and $N=x\, P_1$ for suitable integers $x$ and $y$ with $y\geq 0$. 
Then we apply Proposition \ref{differenceLN}
and using that $C_1\cdot U=2(g_1+d_1-1)$ we find the following.
\begin{lemma} If we write $F_{A}-F_{A_1}=(f_{A}-f_{A_1}) F_{\Sigma}$
with $F_{\Sigma}$ the fibre over $\Sigma$ we have that 
${\rm ord}_{\Sigma}\mathfrak{m}_{\rm st}-
{\rm ord}_{\Sigma}\mathfrak{m}_{\mathcal{L},N}$ is equal to
$$
(f_{A}-f_{A_1}) +
 (x+k-\frac{y+1}{2}) \, y \, d_1 -\frac{1}{2} x(x-1)\,  (d-1) +(1-g_1)y-x\, .
$$ 
\end{lemma}
In certain cases this is positive. For example, if one takes $g_1=1$ where this 
parametrizes curves with an elliptic tail. Then we set $y=k$ and $x=0$.
We have $A_1=P_1-\tilde{\pi}^*Z=P_1-kd_1P_1$ and $F_A=0$ and also $f_{A_1}=1-kd_1$. 
Then the above contribution
is $k(k+1)d_1/2-1$ and this is positive. Therefore the divisor class found by
Patel is larger than that of the Zariski closure of the Maroni locus. 

\bigskip

Next we compare with the results of Deopurkar and Patel for the trigonal case.
In this case the boundary divisors of $\overline{\mathcal{H}}_{3,g}$ are
listed in \cite[p.\ 872]{D-P}. We have the following divisors:
$\Delta$, corresponding to a generic irreducible trigonal curve with a node,
$H$, corresponding to 
the generic hyperelliptic curve, and divisors $\Delta_i(g_1,g_2)$
for $i=1,\ldots,6$. We refer to the description given in 
loc.\ cit. In that paper the formula 
$$
\hbox{\rm Residual}=(7g+6)\lambda -g \delta -2(g-3) \mu
$$
(where $\mu$ refers to the Zariski closure of the Maroni locus and
$\delta $ is the class of the pull back of the total
 boundary of $\overline{\mathcal{M}}_g$)
given on page 877 
can be used to calculate the coefficients of the boundary 
divisors in the expression for the closure of the Maroni locus. For example, take the case
of $\Delta_1(g_1,g_2)$ with $g_1+g_2=g-2$. The number of 
branch points on the base $P_1$ (resp.\ $ P_2$) is $2(g_1+2)$ (resp.\
$2(g-g_1)$) and the pull back of $\delta$ is $3\, \Delta_1$. If $g_1$ is even their formula
gives
$$
(7g+6)\lambda -3g-2(g-3)\mu=\frac{3}{2} g_1(g-g_1-2)
$$
and using the contribution $(g_1+2)(g-g_1)/2(2g+3)$ 
to the Hodge class $\lambda$
we find the contribution of the Maroni closure $\mu=\lambda/2$. 
Comparing with our formula for the standard Maroni extension 
$\mathfrak{m}_{\rm st}$
given in Theorem \ref{classMaroni} we get $r=0$, $c=0$, 
hence $\sigma= \lambda/2$ and the 
results agree. Similarly, if $g_1$ is odd they find $\mu=\lambda/2-1/4$
and our Theorem \ref{classMaroni} gives with $r=1$, $d-n=0$ and $c=-2$ that 
$\sigma= \lambda/2-1/4$. So in the case of the boundary divisors of type 
$\Delta_1$ the two results agree.
However, for some boundary divisors the result of Theorem \ref{classMaroni}
for the standard class $\mathfrak{m}_{\rm st}$ 
gives a higher multiplicity than the result of 
Deopurkar and Patel. We list these cases in a table.

\smallskip
\vbox{
\bigskip\centerline{\def\quad{\hskip 0.6em\relax}
\def\quod{\hskip 0.5em\relax }
\vbox{\offinterlineskip
\hrule
\halign{&\vrule#&\strut\quod\hfil#\quad\cr
height2pt&\omit&&\omit&&\omit&\cr
& div && $(g_1,g_2)$ && difference&\cr
\noalign{\hrule}
& $\Delta$ && && $0$ & \cr
& $\Delta_1$ && && $0$ & \cr
& $\Delta_2$ && && $0$ & \cr
& $\Delta_3$ && $g_1\equiv 0 (\bmod 2)$  && $0$ & \cr  
& $\Delta_3$ && $g_1\equiv 1 (\bmod 2)$  && $1$ & \cr  
& $\Delta_4$ && $g_2\equiv 0 (\bmod 2)$ && $(g_2+1)^2/4 -1/4$ & \cr  
& $\Delta_4$ && $g_2\equiv 1 (\bmod 2)$ && $(g_2+1)^2/4$ & \cr  
& $\Delta_5$ &&  && $g_2(g_2+1)/2$ & \cr  
& $\Delta_6$ && $g_2 \equiv 0 (\bmod 2)$  && $(g_2+1)^2/4-1/4$ & \cr  
& $\Delta_6$ && $g_2 \equiv 1 (\bmod 2)$  && $(g_2+1)^2/4$ & \cr  
& $H$ && && $g(g+2)/4$ & \cr
} \hrule}
}}

\centerline{Table 2.}
\smallskip

But here the methods given in Sections \ref{twistingV} and \ref{varyingL},
or more specifically \ref{specialcase}, 
that produce Maroni classes $\mathfrak{m}_{\mathcal{L},N}$, 
come to the rescue
and by using this we retrieve the
result of Deopurkar and Patel for the trigonal case as we now illustrate.

The first case where our divisor $\mathfrak{m}_{\rm st}$ 
is not optimal is the case of $\Delta_3$ and $g_1$ odd 
where we get an additional coefficient $1$ for $\Delta_3(g_1,g_2)$.
Here $m=3$, $r=1$ and $c=0$, $(\delta_0,\delta_1,\delta_2,\delta_3)=(0,2,2,0)$
and $A=E_1+E_2$. 
But then we apply the results of Section \ref{twistingV} and take $N=P_1$; 
we see that the maximum is assumed in the point 
$(n_0,n_1,n_2,n_3)=(3/2,1/2,0,0)$ and we get 
according to the formula of Theorem \ref{correction1} that
$\sigma=1$, see Table 1 there 
(with $j=2g_1+2\equiv 0 (\bmod 4)$ and $\mu=[3]$). 
This gives the desired correction.

In the remaining cases $\Delta_4,\Delta_5, \Delta_6$ and $H$ we apply the
method of section \ref{specialcase} and get as reduction of the coefficient
$\sigma_{j,\mu}$ a quadratic polynomial in $b-j$ 
and in all these cases this is exactly equal to the difference in Table~2.
We illustrate this by the case $\Delta_4$. 
We have $g_1+g_2=g-1$ and $\mu=(1,1,1)$ 
so that $d=n=3$ and $m=1$; moreover
$j=2(g-g_2+1)$ and $l=2(g_2+1)$. Substituting this 
in Theorem \ref{correction2} and comparing with the standard class 
$\mathfrak{m}_{\rm st}$ gives the correction term $(g_2+1)^2/4$ 
or $g_2(g_2+2)/4$ as desired. 

The cases $\Delta_5$ and $\Delta_6$ follow similarly 
and the case $H$ is the same as $\Delta_6$ with $g_2=g$ even.

\begin{conclusion}
In the trigonal case ($d=3$ and $g$ even)
by choosing for each $\Sigma$ the line bundles $\mathcal{L}$ and 
$N$ as in Theorem \ref{correction1}
and \ref{correction2} we get 
an effective divisor $\mathfrak{m}_{\mathcal{L}, N}$ on $\overline{\mathcal{H}}_{3,g}$
that equals the Zariski closure of the Maroni locus on $\mathcal{H}_{3,g}$.
\end{conclusion}

\end{section}

 \end{document}